\documentclass[12pt,fleqn]{article}

\usepackage[english]{babel}
\usepackage{makeidx}
\usepackage{latexsym,amsfonts,amssymb,amsmath,longtable,amsthm}
\input amssym.def
\usepackage{latexsym}
\usepackage[all]{xy}
\usepackage{amsfonts}
\usepackage{amsmath}
\usepackage{mathrsfs}
\usepackage{color}
\usepackage{ulem}
\usepackage[table]{xcolor}
\usepackage[unicode,breaklinks=true,colorlinks=true]{hyperref}

\catcode`@=11
\newbox\tr@tto
\setbox\tr@tto=\hbox{{\count0=0\dimen0=-
,9pt\dimen1=1,1pt\loop\ifnum\count0<11 \advance \count0 by1 \vrule
width.51pt height\dimen1
     depth\dimen0\kern-0.17pt\advance\dimen0 by-
0.05pt\advance\dimen1
     by0.1pt\repeat \loop\ifnum\count0<21\advance \count0 by1
\vrule
     width.6pt height\dimen1 depth\dimen0\kern-0.2pt
\advance\dimen0
     by-0.1pt\advance\dimen1 by 0.05pt\repeat}}
\def\medint{\displaystyle\copy\tr@tto\kern-10.4pt\int}
\catcode`@=12

\def\Xint#1{\mathchoice
   {\XXint\displaystyle\textstyle{#1}}%
   {\XXint\textstyle\scriptstyle{#1}}%
   {\XXint\scriptstyle\scriptscriptstyle{#1}}%
   {\XXint\scriptscriptstyle\scriptscriptstyle{#1}}%
   \!\int}
\def\XXint#1#2#3{{\setbox0=\hbox{$#1{#2#3}{\int}$}
     \vcenter{\hbox{$#2#3$}}\kern-.5\wd0}}

\def\dashint{\Xint-}

\newcommand{\R}{{\mathbb R}}

\newcommand{\kk}{{\tilde k}}

\newcommand{\N}{{\mathbb N}}
\newcommand{\Z}{{\mathbb Z}}

\newcommand{\e}{{\varepsilon}}
\newcommand{\al}{{\alpha}}
\newcommand{\Le}{{\mathscr L}}
\newcommand{\GG}{{\mathscr G}}
\newcommand{\E}{\mathop{\rm Ent}}
\newcommand{\M}{{\mathcal  M}}
\newcommand{\Be}{{\mathcal  B}}

\renewcommand{\H}{{\mathcal H}}
\newcommand{\rank}{{\mathrm{rank\,}}}

\newcommand{\bg}{{\bar g}}

\newcommand{\p}{{p_\circ}}

\renewcommand{\b}{{q_\circ}}

\renewcommand{\t}{\tau}
\renewcommand{\tt}{\tau_*}

\newcommand{\loc}{{\rm loc}}

\newcommand{\meas}{\mathop{\rm meas}}
\newcommand{\diam}{\mathop{\rm diam}}

\newcommand{\LL}{\mathrm{L}}
\newcommand{\WW}{\mathrm{W}}

\newcommand{\CC}{\mathrm{C}}

\newcommand{\Cc}{C^{k,\alpha}}

\newcommand{\Ll}{\Le^{k+\al}_{p}}
\newcommand{\Lb}{\mathcal L}
\newcommand{\Lll}{\Le^{k+\al}_{p,1}}
\newcommand{\LK}{\Le^{\al}_{p}}
\newcommand{\LLL}{\Le^{\al}_{p,1}}

\newcommand{\BV}{\mathrm{BV}}

\newtheorem{ttt}{\bf Theorem}[section]

\newtheorem{lem}{\bf Lemma}[section]
\newtheorem{cor}{\bf Corollary}[section]

\theoremstyle{remark}
\newtheorem{rem}{\bf Remark}[section]

\numberwithin{equation}{section}

\newcommand{\dd}{{\rm d}}

\title{ON SOME UNIVERSAL MORSE--SARD TYPE THEOREM}
\author{Adele Ferone, \, Mikhail V.~Korobkov,  \, and \, Alba Roviello}

\begin{document}

\maketitle
\hfill{\it Dedicated to the bright memory of Jean Bourgain,  who inspired this area of research}
\begin{abstract}
\medskip
The classical Morse--Sard theorem claims  that for a mapping $v
\colon \R^n\to\R^{m+1}$ of class $\CC^k$ the measure of critical
values $v(Z_{v,m})$ is zero under condition $k
\ge n-m$. Here the critical set, or
$m$-critical set is defined as $Z_{v,m} = \{ x \in \R^n : \, \rank
\nabla v(x)\le m \}$. Further Dubovitski\u{\i}   in 1957 and
independently Federer and Dubovitski\u{\i}  in 1967 found some
elegant extensions of this theorem to the case of other (e.g.,
lower) smoothness assumptions. They also established the sharpness
of their results within the $\CC^k$ category.

Here we formulate and prove a \textit{bridge theorem} that
includes all the above results as particular cases: namely, if a
function $v:\R^n\to\R^d$ belongs to the Holder class
$C^{k,\alpha}$, $0\le\al\le1$, then for every $q>m$ the identity
$$\H^{\mu}(Z_{v,m}\cap v^{-1}(y))=0$$ holds for $\H^q$-almost
all~$y\in\R^d$, where $\mu=n-m-(k+\alpha)(q-m)$.

Intuitively, the sense of this bridge theorem is very close to
{\it Heisenberg's uncertainty principle} in theoretical physics:
the more precise is the information we receive on measure of the image of the critical set, the less precisely the preimages are described, and vice versa.

The result is new even for the classical $C^k$-case (when
$\al=0$\,); similar result is established for the Sobolev classes
of mappings $W^k_p(\R^n,\R^d)$ with minimal integrability
assumptions $p=\max(1,n/k)$, i.e., it guarantees in general only
{\it the continuity} (not everywhere differentiability) of
a~mapping.
However, using some $N$-properties for Sobolev mappings, established in our previous paper,  we obtained that the sets of
nondifferentiability points of Sobolev mappings are fortunately
negligible in the above bridge theorem. We cover also the case of fractional Sobolev spaces.

The proofs of the most results are based on our previous joint
papers with J.~Bourgain and J.~Kristensen  (2013, 2015). We also
crucially use very deep Y.~Yomdin's entropy estimates of near
critical values for polynomials (based on algebraic geometry
tools).
\medskip

\noindent {\bf MSC 2010:} { 58C25 (26B35 46E30)}

\noindent {\bf Key words:} {\it Sobolev--Lorentz mappings, Bessel potential spaces, Holder
mappings, Morse--Sard theorem,
Dubovitski\u{\i}--Federer theorems}

\end{abstract}


\section{Introduction}\label{Introd}
The Morse--Sard theorem in its classical form states that the
image of the set of critical points of a $\CC^{n-d+1}$ smooth
mapping $v\colon \R^n \to \R^{d}$ has zero Lebesgue measure in
$\R^{d}$. More precisely, assuming that $n \geq d$, the set of
critical points for $v$ is $Z_v=\{x\in\R^n\,:\,\rank \nabla
v(x)<d\}$ and the conclusion is that
\begin{equation}\label{classical}
\Lb^{d}(v(Z_v))=0.
\end{equation}
The theorem was proved by Morse~\cite{Mo} in the case $d=1$ and
subsequently by Sard~\cite{S} in the general vector--valued case.
The celebrated results of Whitney~\cite{Wh} show that the
$\CC^{n-d+1}$ smoothness assumption on the mapping $v$ is sharp.
However, the following result gives valuable information also for
less smooth mappings. \vspace{2mm}

\noindent {\bf Theorem A
(Dubovitski\u{\i}~1957~\cite{Du})}.\label{AA} {\sl Let $n,d,k \in
\N$, and let
 $v \colon \R^n\to\R^{d}$ be a
$\CC^k$--smooth mapping. Put $\nu=n-d-k+1$. Then
\begin{equation}\label{dub1}
\H^\nu (Z_{v}\cap v^{-1}(y))=0\qquad\mbox{ for a.a. }y\in\R^{d},
\end{equation}
where  $\H^\nu$ denotes the $\nu$--dimensional  Hausdorff measure.} \vspace{2mm}

Here and in the following we interpret $\H^\beta$ as the counting
measure when $\beta\le0$. Thus for $k\ge n-d+1$ we have $\nu
\le0$, and $\H^\nu$ in (\ref{dub1}) becomes simply the counting
measure, so the Dubovitski\u{\i} theorem contains the Morse--Sard
theorem as particular case\footnote{It is interesting to note that
this first Dubovitski\u{\i} theorem remained almost unnoticed by
West mathematicians for a long time; another proof was given in
the recent paper~\cite{BHS} covering also some extensions to the
case of H\"{o}lder spaces; see also~\cite{HZ} for the Sobolev
case.}.

A few years later and almost simultaneously,
Dubovitski\u{\i}~\cite{Du2} in 1967 and Federer \cite[Theorem
3.4.3]{Fed} in 1969\footnote{Federer announced~\cite{Fed66} his
result in 1966, this announcement (without any proofs) was sent
on~08.02.1966. For the historical details, Dubovitski\u{\i} sent
his paper~\cite{Du2} (with complete proofs) a month earlier,
on~10.01.1966.} published another important generalization of the
Morse--Sard theorem. \vspace{2mm}

\noindent {\bf Theorem B (Dubovitski\u{\i}--Federer)}.\label{BB}
{\sl For $n,k,d \in \N$ let $m\in\{0,\dots,\min (n,d)-1\}$ and
 $v \colon \R^n\to\R^d$ be a
$\CC^k$--smooth mapping. Put
$q_\circ=m +\frac{n-m }k$. Then
\begin{equation}\label{fed1}
\H^\b(v(Z_{v,m}))=0,
\end{equation}
where $Z_{v,m}$ denotes the set of $m$--critical points of $v$
defined as
$$Z_{v,m}=\{x\in\R^n\,:\,\rank \nabla v(x)\le m\}.$$} \vspace{2mm}

In 2001 Moreira 
\cite{CGTAM} extend the last result to the Holder class~$C^{k,\al}$, i.e., i.e. he proved that for a~mapping $v\in C^{k,\al}(\R^n,\R^d)$ the equality 
(\ref{fed1}) holds with~$\b=m +\frac{n-m}{k+\al}$. 

In view of the wide range of applicability of the above results it
is a natural and compelling problem to extend the results to the classes of Sobolev mappings. 

In the recent paper \cite{HKK} by Haj\l{}asz P., Korobkov M.V., and Kristensen J.
for $k\le n$ and for Sobolev
classes $W^k_p(\R^n,\R^d)$ it was  proved a~\textit{bridge theorem}
that includes Theorems~A--B  as particular cases (see below
Theorem~\ref{DFT-F}-(ii)). In the present paper we extend this result
for the Holder classes $\Cc$ and for Sobolev spaces
$W^k_p(\R^n,\R^d)$ with arbitrary integer~$k\ge1$, and also for
fractional Sobolev spaces $\Le^{k+\al}_p$ 
(e.g., for Bessel potential spaces; see Theorems~\ref{DST-H}--\ref{DFT-F}\,).

The integrability assumptions here are very minimal and sharp, they are of kind \linebreak
$p(k+\al)\ge n$, i.e., they
guarantee in general only
{\it the continuity} (not everywhere differentiability) of
a~mapping. However, we proved that the 'bad' set of nondifferentiability points of Sobolev mappings is fortunately
negligible in the above bridge theorem (see Theorem~\ref{DFT-F-negl}\,) because of some Luzin type $N$--properties with respect to lower dimensional
Hausdorff measures established in our previous papers~\cite{BKK2,FKR-n,KK15}.

Let us note, in the conclusion, that the Morse--Sard theorem for the Sobolev spaces was very fruitful in mathematical fluid mechanics, in particular,
it was used in the recent  solution of the so-called Leray's problem for the steady Navier--Stokes system (see~\cite{KPR}\,).

\subsection{Bridge F.-D.-theorems for the Holder classes of mappings}

We  say that a~mapping $v:\R^n\to\R^d$ belongs to the class $\Cc$
for some positive integer $k$ and $0<\alpha\le1$ if there exists
a~constant $L\geq 0$ such that
\begin{equation*}\label{hic1}\mbox{$|\nabla^kv(x)-\nabla^kv(y)|\le
L\,|x-y|^\alpha$ \qquad for all $x,y\in\R^n$.}
\end{equation*}

To simplify the notation, let us make the following agreement: for
$\alpha=0$ we identify~$C^{k,\alpha}$ with usual spaces of 
$C^k$-smooth mappings. The following theorem is one of the main
results of the paper.

\begin{ttt}\label{DST-H}{\sl Let
$m\in\{0,\dots,n-1\}$, \,$k\ge0$, \,$0\le\al\le1$, \,$k+\al\ge1$, \,$ d>m$, \,and
\,$v\in \Cc(\R^n,\R^d)$. Then for any $q\in (m ,\infty)$ the
equality
\begin{equation*}\label{hf1}
\H^{\mu_q}(Z_{v,m}\cap v^{-1}(y))=0\qquad\mbox{ for \ $\H^q$-a.a.
}y\in\R^d
\end{equation*}
holds, where
\begin{equation*}\label{hf2}\mu_q=n-m -(k+\alpha)(q-m ),
\end{equation*} and
$Z_{v,m}$ denotes the set of $m$-critical points of~$v$:\ \
$Z_{v,m}=\{x\in\R^n\,:\,\rank\nabla v(x)\le m \}$.}
\end{ttt}

Let us note, that for the classical $C^k$-case, i.e., when
$\alpha=0$, the behaivior of the function $\mu_q$ is very natural:
\begin{equation*}\label{dub3-q}
\begin{array}{lcr}
\mu_q=0\qquad\mbox{ for }q=\b=m +\frac{n-m }k\ \ \qquad\mbox{(Dubovitski\u{\i}--Federer Theorem~B)};\\
[13pt]
\mu_q<0 \qquad\mbox{ for }q>\b\ \ \ \qquad\mbox{[ibid.]};\\
[13pt] \mu_q=\nu \qquad\mbox{ for }q=m+1\ \ \qquad\mbox{(Dubovitski\u{\i} Theorem~A)}; \\
[13pt]\mu_q=n-m \qquad\mbox{ for }q=m .\end{array}
\end{equation*}
The last value cannot be improved in view of the trivial example
of a linear mapping $L\colon \R^n\to\R^d$ of rank $m $.

Thus, Theorem~\ref{DST-H} contains all the previous theorems
(Morse--Sard, A,\,B\, and even the Bates theorem for
$C^{k,1}$-Lipschitz functions~\cite{Bates}) as particular cases.

Intuitively, the sense of the~Bridge Theorem~~\ref{DST-H} is very
close to the {\it Heisenberg's uncertainty principle} in
theoretical physics: the more precisely information we received on
measure of the image of the critical set, the less precisely the
preimages are described, and vice versa.

\begin{rem}\label{bhs}
As we mentioned before, for $q=\b=m +\frac{n-m }{k+\alpha}$ and $\mu_q=0$ (as in the Dubovitski\u{\i}-Federer Theorem~B\,) the assertion
of Theorem~\ref{DST-H} was proved in 2001 in the paper of Moreira~\cite{CGTAM}. \,For the minimal rank value $m=0$ (i.e., when the gradient totally vanishes on the critical set) and
$q=\b=\frac{n}{k+\alpha}$, \,$\mu_q=0$,  the assertion of
Theorem~\ref{DST-H} was proved by Kucera~\cite{Kucera} in 1972. 
Further, for partial case $q=m+1=d$ (as in the Dubovitski\u{\i}
theorem~A\,)  and under additional assumption that
\begin{equation}\label{u-ass}
\mbox{$|\nabla^kv(x)-\nabla^kv(y)|\le\omega\bigl(|x-y|\bigr)\cdot
|x-y|^\alpha$ \ \ with \ $\omega(r)\to0$ \ as $r\to0$},
\end{equation}
the assertion
of Theorem~\ref{DST-H} was proved in the paper  Bojarski~B. et al.~\cite{BHS} in 2005. 
Under the same asymptotic assumption~(\ref{u-ass})  the above Moreira result (i.e., when $q=\b$,
\,$\mu_q=0$\,)   was proved by Yomdin in the
paper~\cite{Yom} in 1983. 
\end{rem}

\medskip

\subsection{Bridge F.-D.-theorems for mappings of Sobolev and fractional Sobolev spaces}\label{bst}

Let $k\in\N$, $1<p<\infty$  and $0\le \alpha<1$. One of the most natural type of
 fractional Sobolev spaces  is
{\it (Bessel) potential spaces} $\Le^{k+\alpha}_p$. (They are  Sobolev analog of classical Holder classes $\Cc$.\,)

Recall, that a~function $v:\R^n\to\R^d$ belongs to the
space $\Le^{k+\alpha}_p$, if it is a convolution of a function~$g\in L_p(\R^n)$ with the Bessel
kernel~$G_{k+\alpha}$, where
$\widehat{G_{k+\alpha}}(\xi)=(1+4\pi^2\xi^2)^{-(k+\alpha)/2}$. It is well known that for the integer exponents (i.e., when $\al=0$) one has the identity
\begin{equation}\label{fN1} \Le^{k}_p(\R^n)=W^k_p(\R^n)\qquad\mbox{ if }\ \quad 1<p<\infty,
\end{equation}
where $W^k_p(\R^n)$ is the classical Sobolev space consisting of functions whose generalised derivatives up to order $\le k$ belongs to the Lebesgue space~$L_p(\R^n)$. 

As usual, if  $(k+\al)p>n$, then functions from the potential space $\Le^{k+\alpha}_p(\R^n)$ are continuous by Sobolev Theorem.
But if $(k+\al)p=n$, then functions from potential spaces~$\Le^{k+\alpha}_p(\R^n)$ are discontinuous in general. 
Thus for this limiting case we need to consider the Bessel--Lorentz potential space~$\Le^{k+\alpha}_{p,1}(\R^n)$ to have
the~continuity. Namely, $\Le^{k+\alpha}_{p,1}(\R^n,\R^d)$ denotes the space
of functions which could be represented as a~convolution of
the~Bessel potential~$G_{k+\al}$ with a function~$g$ from the Lorentz
space~$L_{p,1}$ (see the definition of these spaces in the
section~\ref{prel}). Similarly to~(\ref{fN1}),  
for the integer exponents (i.e., when $\al=0$) one has the identity
\begin{equation}\label{fN1-L} \Le^{k}_{p,1}(\R^n)=W^k_{p,1}(\R^n)\qquad\mbox{ if }\ \quad 1<p<\infty,
\end{equation}
where $W^k_{p,1}(\R^n)$ consists of all functions $v\in W^k_{p}(\R^n)$ whose partial derivatives of order~$k$ 
belongs to the~Lorentz
space~$\LL_{p,1}$ (see, e.g., ~\cite{FKR-n}\,). 

   \begin{ttt}\label{DFT-F}{\sl Let
$m\in\{0,\dots,n-1\}$, \,$k\ge1$, \,$d>m$, \,$0\le\al<1$, \,$p\ge1$
\,and \,let $v:\R^n\to \R^d$ be a~mapping
for which one of the following cases holds:

\begin{itemize}
\item[(i)] \,$\al=0$, \,$kp>n$, \,and \,$v\in W^k_p(\R^n,\R^d)$;

\item[(ii)] \,$\al=0$, \,$kp=n$, \,and \,$v\in W^k_{p,1}(\R^n,\R^d)$;

\item[(iii)] \,$0<\al<1$, \,$p>1$, \,$(k+\al)p>n$, \,and \,$v\in \Le^{k+\al}_{p}(\R^n,\R^d)$;

\item[(iv)] \,$0<\al<1$, \,$p>1$, \,$(k+\al)p=n$, \,and \,$v\in \Le^{k+\al}_{p,1}(\R^n,\R^d)$.
\end{itemize}
Then the mapping~$v$ is continuous and for any $q\in (m,\infty)$ the equality
\begin{equation*}\label{sob-hf1-hhh}
\H^{\mu_q}(Z_{v,m}\cap v^{-1}(y))=0\qquad\mbox{ for \ $\H^q$-a.a.
}y\in\R^d
\end{equation*}
holds, where again
\begin{equation*}\label{sob-hf2-hhh}\mu_q=n-m-(k+\alpha)(q-m),
\end{equation*} and
$Z_{v,m}$ denotes the set of $m$-critical points of~$v$:  
$Z_{v,m}=\{x\in\R^n\setminus A_v:\rank\nabla v(x)\le m\}$.
}
\end{ttt}

Here $A_v$ means  the set of `bad' points at which either the function~$v$
is not differentiable or which are not the  Lebesgue points for~$\nabla v$.
Recall, that by approximation results (see, e.g.,
\cite{Sw} \,and \,\cite{KK15}\,) under conditions of Theorem~\ref{DFT-F} the equalities
\begin{eqnarray*}\label{hf0}\!\!\!\!\!\!\!\!\!\!\!\!\!\!\!\!\!\!\!\H^\tau(A_v)=0\qquad\forall\tau>\tt:=n-(k+\al-1)p\qquad\mbox{in cases (i), (iii)};\\
\label{hf0--}\!\!\!\!\!\!\!\!\H^{\tt}(A_v)=\H^p(A_v)=0\qquad\tt:=n-(k+\al-1)p=p\qquad\mbox{in cases (ii), (iv)}
\end{eqnarray*}
are valid (in particular, $A_v=\emptyset$ if $(k+\al-1)p>n$).
However, it was proved in~\cite{FKR-n} that the impact of the "bad"\ set~$A_v$ is negligible
in the~ Bridge D.-F. Theorem~\ref{DFT-F}, i.e., the following statement holds:

\begin{ttt}[\cite{FKR-n}]\label{DFT-F-negl}{\sl Let the conditions of Theorem~\ref{DFT-F} be fulfilled for a~function $v:\R^n\to\R^d$. Then
\begin{equation*}\label{dub7-qqqq} \H^{\mu_q}(A_v\cap
v^{-1}(y))=0\qquad\mbox{ for \ $\H^q$--a.a. }y\in\R^d
\end{equation*}
for any $q>m$.
}
\end{ttt}

\

\begin{rem}\label{remq0}Note, that since $\mu_q\le0$ \,for \,$q\ge\b=m +\frac{n-m}{k+\alpha}$, the assertions of
Theorems~\ref{DFT-F}--\ref{DFT-F-negl} are equivalent to the equality
$0=\H^q\bigl[v(A_v\cup Z_{v,m})\bigr]$ \ for \ $q\ge\b$, so
it is sufficient to check the assertions of
Theorems~\ref{DFT-F}--\ref{DFT-F-negl} for
\,$q\in(m,\b]$ only.
\end{rem}

\begin{rem}\label{hist1}
Note that in the pioneering paper by De~Pascale~\cite{DeP} the assertion of the~initial Morse--Sard theorem~(\ref{classical}) \ (i.e., when $k=n-m$, $q=\b=m+1=d$, $\mu_q=0$\,) was obtained for the Sobolev
classes $W^k_p(\R^n,\R^d)$ under additional assumption $p>n$ (in this case the classical embedding $W^k_p(\R^n,\R^d)\hookrightarrow C^{k-1}$ holds, so there are no problems with nondifferentiability points). For the same Sobolev class $W^k_p(\R^n,\R^d)$ with $p>n$ the assertion of the Dubovitski\u{\i} Theorem~A was proved in the recent paper~\cite{HZ} by 
P.~Haj\l{}asz  and S.~Zimmermann.\, Finally, 
the assertion of Bridge Theorem~\ref{DFT-F}-(ii) was proved in our previous paper~\cite{HKK} with P.~Haj\l{}asz and J.~Kristensen\footnote{The~only technical difference is that 
in~\cite{HKK}   we used the notation $Z_{v,m}=\{x\in\mathbb R^n\setminus A_v:\mathrm{rank\,}\nabla v(x)<m\}$, i.e., there $m-1$ plays the role of the~parameter~$m$ of the~present article.}. 
\end{rem}

\

In conclusion, let us comment briefly that the merge ideas for the proofs
are from our previous papers~\cite{BKK2}, \cite{KK3,KK15}
and~\cite{HKK}. In particular, the joint papers \cite{BKK,BKK2} by
one of the authors with J.~Bourgain and J.~Kristensen contain many of the key ideas
that allow us to consider nondifferentiable Sobolev mappings. As in \cite{BKK2} (and
subsequently in \cite{KK3}) we also crucially use Y.~Yomdin's (see
~\cite{Yom}) entropy estimates of near critical values for
polynomials (recalled in Theorem~\ref{lb8} below). These Yomdin's results seems to be very deep and fruitful in the topic,
see, e.g., the~very recent paper~\cite{Barbet} where the  Morse-Sard theorems were proved for min-type functions and for
Lipschitz selections.

In addition to the above mentioned papers there is a growing
number of papers on the topic, including
\cite{Alb,AS,Bates,Bu,Fig,HZ,Nor,PZ,Putten1}.
\bigskip

Some words about the structure of the paper. In the second section we give some basic  definitions and recall some classical results in analysis, which are very useful tools in our study. In the third sections we give the proofs of main theorems  formulated above. For a reader convenience, the most technical part is moved to the last section
Appendix~\ref{appendixII}, where we obtain estimates for the critical values on a~single $n$-dimensional cube. These estimates are strong enough and useful for a~solution of
 the following more general

 \medskip
 {\bf Problem~C.} {\sl Let $S$ be a subset of critical set~$Z_{v,m}=\{\rank \nabla v\le m\}$ and the equality $\H^\tau(S)=0$ (or the~ inequality~ $\H^\tau(S)<\infty$\,) holds for some
 $\t>0$. Does it imply that $\H^{\sigma}(v(E))=0$ for some $\sigma=\sigma(\tau)$?}

\medskip
 The complete solution to this problem is done in our new paper~\cite{FKR-3}; this solution  based on the technique developed in  the present paper.

\medskip
\noindent {\em Acknowledgment.} M.K. was partially supported by
the Ministry of Education and Science of the Russian Federation
(grant 14.Z50.31.0037). The main part of the paper was written
during a visit of M.K. to the University of Campania "Luigi
Vanvitelli" in~2017, and he is very thankful for the hospitality.

\section{Preliminaries}
\label{prel}

\noindent

\noindent By  an {\it $n$--dimensional interval} we mean a closed
cube in $\R^n$ with sides parallel to the coordinate axes. If $Q$
is an $n$--dimensional cubic interval then we write $\ell(Q)$ for
its sidelength.

For a subset $S$ of $\R^n$ we write $\Lb^n(S)$ for its outer
Lebesgue measure (sometimes we use the symbol~$\meas S$ for the
same purpose\,). The $m$--dimensional Hausdorff measure is denoted
by $\H^m$ and the $m$--dimensional Hausdorff content by
$\H^{m}_{\infty}$. Recall that for any subset $S$ of $\R^n$ we
have by definition
$$
\H^m (S)=\lim\limits_{t\searrow 0}\H^m_t (S) = \sup_{t >0}
\H^{m}_{t}(S),
$$
where for each $0< t \leq \infty$,
$$
\H^m_t (S)=\inf\left\{ \sum_{i=1}^\infty(\diam S_i)^m\ :\ \diam
S_i\le t,\ \ S \subset\bigcup\limits_{i=1}^\infty S_i \right\}.
$$
It is well known that $\H^n(S)  = \H^n_\infty(S)\sim\Lb^n(S)$ for
sets~$S\subset\R^n$.

To simplify the notation, we write $\|f\|_{\LL_p}$ instead of
$\|f\|_{\LL_p(\R^n)}$, etc.

The Sobolev space $\WW^{k}_{p} (\R^n,\R^d)$ is as usual defined as
consisting of those $\R^d$-valued functions $f\in \LL_p(\R^n)$
whose distributional partial derivatives of orders $l\le k$ belong
to $\LL_p(\R^n)$ (for detailed definitions and differentiability
properties of such functions see, e.g., \cite{EG}, \cite{M},
\cite{Ziem}, \cite{Dor}). Denote by $\nabla^k f$ the vector-valued
function consisting of all $k$-th order partial derivatives of $f$
arranged in some fixed order. However, for the case of first order
derivatives $k=1$ we shall often think of $\nabla f(x)$ as the
Jacobi matrix of $f$ at $x$, thus the $d \times n$ matrix whose
$r$-th row is the vector of partial derivatives of the $r$-th
coordinate function.

We use the norm
$$
\|f\|_{\WW^{k}_{p}}=\|f\|_{\LL_p}+\|\nabla
f\|_{\LL_p}+\dots+\|\nabla^kf\|_{\LL_p},
$$
 and unless otherwise specified all norms on the
spaces $\R^s$ ($s \in \N$) will be the usual euclidean norms.

Working with locally integrable functions, we always assume that
the precise representatives are chosen. If $w\in
L_{1,\loc}(\Omega)$, then the precise representative $w^*$ is
defined for {\em all} $x \in \Omega$ by
\begin{equation*}
\label{lrule}w^*(x)=\left\{\begin{array}{rcl} {\displaystyle
\lim\limits_{r\searrow 0} \dashint_{B(x,r)}{w}(z)\,\dd z}, &
\mbox{ if the limit exists
and is finite,}\\
 0 \qquad\qquad\quad & \; \mbox{ otherwise},
\end{array}\right.
\end{equation*}
where the dashed integral as usual denotes the integral mean,
$$
\dashint_{B(x,r)}{ w}(z) \, \dd
z=\frac{1}{\Lb^n(B(x,r))}\int_{B(x,r)}{ w}(z)\,\dd z,
$$
and $B(x,r)=\{y: |y-x|<r\}$ is the open ball of radius $r$
centered at $x$. Henceforth we omit special notation for the
precise representative writing simply $w^* = w$.

If $k<n$, then it is well-known that functions from Sobolev spaces
$\WW^{k}_{p}(\R^n)$ are continuous for $p>\frac{n}k$ and could be
discontinuous for $p\le \p=\frac{n}k$ (see, e.g., \cite{M,Ziem}). The Sobolev--Lorentz space $\WW^{k}_{\p,1}(\R^n)\subset
\WW^{k}_{\p}(\R^n)$ is a refinement of the corresponding Sobolev
space. Among other things functions that are locally in
$\WW^{k}_{\p,1}$ on $\R^n$ are in particular continuous (see, e.g., \cite{KK3}\,).

Here we only mentioned the Lorentz space $\LL_{p,1}$, and in this
case one may rewrite the~norm as (see for instance
\cite[Proposition 3.6]{Maly2})
\begin{equation*}\label{lor1}
\|f\|_{L_{p,1}}=
\int\limits_0^{+\infty}\bigl[\Lb^n(\{x\in\R^n:|f(x)|>t\})\bigr]^{
\frac1p} \, \dd t.
\end{equation*}
Of course, we have the inequality
\begin{equation}\label{lor1---}
\|f\|_{L_p}\le \|f\|_{L_{p,1}}.\end{equation}

By definition we put $\|g\|_{L_{p,1}(E)}:=\|1_E\cdot g\|_{L_{p,1}}$,
where $1_E$ is the indicator function of~$E$.

Denote by $\WW^{k}_{p,1}(\R^n)$ the space of all functions $v\in
\WW^k_p(\R^n)$ such that in addition the Lorentz norm~$\|\nabla^k
v\|_{\LL_{p,1}}$ is finite.

For a function $f\in L_{1,\loc}(\R^n)$ we often use the classical
Hardy--Littlewood maximal function:
$$
\M f(x)=\sup\limits_{r>0}  \dashint_{B(x,r)} \! |f(y)|\,\dd y.
$$

\subsection{On potential spaces~$\LK$}
To simplify our descriptions, below in the next two subsections we will write $\al$ instead of $k+\al$, so here we assume that $\al\in \R_+$ (i.e., here not necessarily $\al<1$, as in formulations of main results\,).

In the paper we deal with {\it(Bessel)-potential space} $\LK$ with
$\al>0$. Recall, that function $v:\R^n\to\R^d$ belongs to the
space $\LK$, if it is a convolution of the Bessel
kernel~$G_{\alpha}$ with a function~$g\in L_p(\R^n)$:
$$v=\GG_\alpha(g):=G_\al*g,$$
where
$\widehat{G_\alpha}(\xi)=(1+4\pi^2\xi^2)^{-\alpha/2}$. In
particular, \
$$\|v\|_{\LK}:=\|g\|_{L_p}.$$
It is well known that
\begin{equation*}\label{fN1----} \LK(\R^n)=W^\al_p(\R^n)\qquad\mbox{ if }\ \al\in\N\quad\mbox{and}\quad 1<p<\infty.
\end{equation*}

Recall, that the Bessel kernel is radial, $G_\al(x)=G_\al(|x|)$, and it could be calculated as
\begin{equation}\label{bc-2}
G_\al(x)\ =\ a_\al\int\limits_0^\infty t^{\frac{\al-n}2}e^{-\frac{\pi |x|^2}{t}-\frac{t}{4\pi}}\frac{dt}t,
, \qquad \forall\al>0,
\end{equation}
where $a_\al$ is some constant.

It is well known, that $G_\al(x)<a_\al\,|x|^{\al-n
}$ for $0<\al<n$ (see, e.g., \cite[page10]{AdH}\,) and we need
some simple  technical extension of this fact to the derivatives.
\begin{lem}\label{lem-bc-2}{\sl
If $0<\al<n+2$,
then
for any integer $j\in\N$ the estimate
\begin{equation}\label{bc-4}
\bigl|\nabla^jG_\al(x)\bigr|\le C\,|x|^{\al-n-j}
\end{equation}
holds, where the constant $C$ depends on~$\al,n,j$ only.}
\end{lem}

\begin{proof}
Denote $$f_\al(r)=\int\limits_0^\infty t^{\frac{\al-n}2}e^{-\frac{\pi r^2}{t}-\frac{t}{4\pi}}\frac{dt}t.$$
 Then by direct calculation
 $$f'_\al(r)=-2\pi r\int\limits_0^\infty t^{\frac{\al-n-2}2}e^{-\frac{\pi ^2}{t}-\frac{t}{4\pi}}\frac{dt}t=-2\pi r^{\al-n-1}\int\limits_0^\infty t^{\frac{\al-n-2}2}e^{-\frac{\pi}{t}-\frac{t r^2}{4\pi}}\frac{dt}t $$
 (see, e.g., \cite[page13]{AdH}). This finishes the proof for $j=1$. The proof for $j>1$ could be produced the same way by induction.
 \end{proof}

\subsection{On Lorentz potential spaces~$\LLL$}

To cover some other limiting cases,
denote by $\LLL(\R^n,\R^d)$ the space
of functions which could be represented as a~convolution of
the~Bessel potential $G_\al$ with a function~$g$ from the Lorentz
space~$L_{p,1}$; respectively,
$$\|v\|_{\LLL}:=\|g\|_{L_{p,1}}.$$

Because of inequality~(\ref{lor1---}), we have an evident inclusion
\begin{equation*}\label{lor2}
\LLL(\R^n)\subset \LK(\R^n).\end{equation*}

\begin{ttt}[see, e.g., Theorem~2.2 in~\cite{FKR-n}, ~cf. with Lemma~3 on page 136 in~\cite{St}]\label{propert-pot-lor1}{\sl
Let $\alpha\ge1$ and $1<p<\infty$. Then $f\in
\LLL(\R^n)$ iff $f\in\Le^{\al-1}_{p,1}(\R^n)$ \ and \ $\frac{\partial f}{\partial x_j}\in \Le^{\al-1}_{p,1}(\R^n)$ for every $j=1,\dots,n$. }
\end{ttt}

(Here for convenience we use the agreement that
$\LK(\R^n)=L_p(\R^n)$ when $\al=0$.\,)

\begin{cor}\label{propert-pot-lor2}{\sl
Let $k\in\N$ and $1<p<\infty$. Then $\Le^k_{p,1}(\R^n)=W^k_{p,1}(\R^n)$, where
$W^k_{p,1}(\R^n)$ is the space of functions such that all its distributional partial derivatives of order~$\le k$ \ belong to~$L_{p,1}(\R^n)$.  }
\end{cor}
Note, that the space $W^k_{p,1}(\R^n)$ admits also a~simpler (but equivalent) description: it consists of functions $f$ from the usual Sobolev space~$W^k_p(\R^n)$  satisfying the additional condition
$\nabla^kf\in L_{p,1}(\R^n)$ (i.e., this condition is on the highest derivatives only), see, e.g.,~\cite{Maly2}.

\subsection{Approximation of Sobolev functions by polynomials}\label{asfp}

For a mapping $u \in \LL_1(Q,\R^d)$, $Q\subset\R^n$, $m\in\N$,
define the polynomial $P_{Q,m}[u] $ of degree at most~$m$ by the
following rule:
\begin{equation*}
\label{0}\int_Qy^\gamma \left( u(y)-P_{Q,m}[u](y) \right) \,\dd
y=0
\end{equation*}
for any multi-index $\gamma=(\gamma_1,\dots,\gamma_n)$ of length
$|\gamma|=\gamma_1+\dots+\gamma_n\le {m}$.

The following well--known bounds will be used on several
occasions.

\begin{lem}[see, e.g., \cite{KK3}]\label{lb3}{\sl
Suppose $v\in \WW^{k}_{1}(\R^n,\R^d)$ { with $k\ge n$}. Then  $ v$ is a continuous mapping and for any
$n$-dimensional cubic interval $Q\subset \R^n$ the estimates
\begin{equation*}
\label{1--} \bigl\|v-P\bigr\|_{L_\infty(Q)}\le C\ell(Q)^{k-n}\,\|
\nabla^{k} v\|_{\LL_{1}(Q)};
\end{equation*}
\begin{equation*}
\label{1} \bigl\|\nabla
\bigl(v-P\bigr)\bigr\|_{L_\infty(Q)}\le C\ell(Q)^{k-1-n}\,\|
\nabla^{k} v\|_{\LL_{1}(Q)}\qquad\mbox{\rm if \ $k\ge n+1$};
\end{equation*}
hold, where $P=P_{Q,k-1}[v]$ \ and $C$ is a constant depending on $n,d,k$ only.
Moreover, the mapping $v_{Q}(y)=v(y)-P(y)$, $y\in Q$, can
be extended from~$Q$ to the entire  $\R^n$ such that the
extension (denoted again) $v_{Q}\in\WW^{k}_{1}(\R^n,\R^d)$ and
\begin{equation}
\label{1'} \|\nabla^{k} v_{Q}\|_{\LL_{1}(\R^n)}\le C_0
\|\nabla^{k} v\|_{\LL_{1}(Q)},
\end{equation}
 where $C_0$ also depends on $n,d,k$ only. }
\end{lem}

\subsection{Approximation of fractional Sobolev functions by polynomials}

We need the following natural estimate whose analogs for Sobolev case are well-known (see, e.g.,~\cite{M}\,).

\begin{ttt}\label{SCE-1}{\sl Let  $k\ge1$, \,$0\le\al< 1$, $1<p<\infty$, \,and
\,$v\in \Ll(\R^n,\R^d)$, i.e.,  $v=\GG_{k+\alpha}(g):=G_{k+\al}*g$ for some $g\in L_{p}(\R^n)$. Suppose in addition that
$$1<k+\al<n+2.$$ Then for any
$n$-dimensional  interval $Q\subset\R^n$ there exists a polynomial~$P=P_Q$ of degree~$k$ such that
the difference  $v_Q=v-P$ satisfies the estimate
\begin{equation}
\label{pol-est-1} |\nabla v_Q(x)|\le C\int\limits_{Q}\frac{\M g(y)}{|x-y|^{n-k-\al+1}}\,dy\ \qquad\forall x\in Q,
\end{equation}
where $r=\ell(Q)$ and  the constant $C$ depends on $n,k,\alpha,d,p$ only.}
\end{ttt}

\begin{proof} Really, this theorem in essence was proved in the paper~\cite{FKR-n}. Let us recall some arguments from there.
Fix an $n$-dimensional  interval $Q\subset\R^n$  and denote by $2Q$ the double cube with the same center
 as~$Q$ of size~$\ell(2Q)=2\ell(Q)$ . We have
 $$v(x)=\int\limits_{\R^n}G_{k+\al}(x-y)\,g(y)\,dy.$$
 Split the
function~$v$ into the sum
\begin{equation} \label{max-ap2---} v=v_1+v_2,
\end{equation}
where
$$v_1(x):=\int\limits_{2Q}g(y)\,G_{k+\al}(x-y)\,dy,\qquad \
v_2(x):=\int\limits_{\R^n\setminus 2Q}g(y)\,G_{k+\al}(x-y)\,dy.$$
As above, denote by $\M g$ the usual Hardy---Littlewood maximal function for~$g$. Then
 from  \cite[Lemma~A.1]{FKR-n}  and from the estimate
\begin{equation}
\label{pr-s-8} |\nabla G_{k+\al}(z)|\le C|z|^{k+\al-n-1}
\end{equation}
(see~Lemma~\ref{lem-bc-2}\,) it follows immediately that
\begin{equation}
\label{pol-est-2} |\nabla v_1(x)|\le C\,\int\limits_{Q}\frac{\M g(y)}{|x-y|^{n-k-\al+1}}\,dy\ \qquad\forall x\in Q.
\end{equation}
 Analogously, from the similar estimate
\begin{equation}
\label{pr-s-10} |\nabla^j G_{k+\al}(z)|\le C|z|^{k+\al-n-j}
\end{equation}
[ibid.] and from  Lemma~A.2 of the paper~\cite{FKR-n} and its proof, applying to the function~$\nabla^kv_2$ with\footnote{That means, that  now our
function~$\nabla^kv_2$ plays the role of~mapping~$v$ from arguments of~\cite[proof of Lemma~A.2]{FKR-n}.}
 parameter~$\theta=1-\al$, we obtain
 \begin{equation}
\label{ll0--} \diam\bigl[\nabla^kv_2(Q)\bigr]\le C\,r^{\al-n}\int\limits_{Q}\M\,g(y)\,dy.
\end{equation}
Take the corresponding approximate polynomial $P=P_Q$ of degree~$k$, then for the difference $\tilde v=v_2-P$ we obtain the following estimates:
\begin{equation}
\label{pr-s-13}   \sup\limits_{x\in Q}|\nabla^k \tilde v(x)|\le C\,r^{\al-n}\int\limits_{Q}\M\,g(y)\,dy,
\end{equation}
\begin{equation}
\label{pr-s-14}    \sup\limits_{x\in Q}|\nabla\tilde v(x)|\le C\,r^{\al-n+k-1}\int\limits_{Q}\M\,g(y)\,dy.
\end{equation}
Evidently,
$$r^{\al-n+k-1}\int\limits_{Q}\M\,g(y)\,dy\,\le\, C\,\int\limits_{Q}\frac{\M g(y)}{|x-y|^{n-k-\al+1}}\,dy\ \qquad\forall x\in Q. $$
 From the last formula and inequalities (\ref{pr-s-14}), (\ref{pol-est-2}) the required estimate~(\ref{pol-est-1}) follows directly.
 \end{proof}

\begin{rem}\label{rem-SCE-1}If under above conditions we have in addition
$(k+\al-1)p>n$, then  by Holder inequality the estimate  (\ref{pol-est-1}) implies
\begin{equation}
\label{pr-s-12}  \sup\limits_{x\in Q}|\nabla  v_Q(x)|\le C\,r^{k+\al-1-\frac{n}p}\|\M g\|_{L_p(Q)}.
\end{equation}
\end{rem}
\subsection{On Yomdin's entropy estimates for the nearcritical values of polynomials}

For a subset $A$ of ${\R}^d$ and $\varepsilon>0$ the
$\varepsilon$--entropy of $A$, denoted by $\E(\varepsilon,A)$, is
the minimal number of closed balls of radius $\varepsilon$
covering $A$. Further, for a linear map~$L\colon \R^n\to\R^d$ we
denote by $\lambda_j(L)$, $j=1,\dots,d$, its singular values
arranged in decreasing order: \ $\lambda_1(L) \ge \lambda_2(L)
\ge\dots\ge\lambda_d(L)$. Geometrically the singular values are
the lengths of the semiaxes of the ellipsoid $L( \partial
B(0,1))$. We recall that the singular values of $L$ coincide with
the eigenvalues repeated according to multiplicity of the
symmetric nonnegative linear map~$\sqrt{LL^{\ast}}\colon \R^d \to
\R^d$. Also for a mapping~$f\colon \R^n\to\R^d$ that is
approximately differentiable at $x \in \R^n$
put~$\lambda_j(f,x)=\lambda_j(d_xf)$, where by $d_xf$ we denote
the  approximate differential of $f$ at~$x$. The next result is
the basic ingredient of our proof.

\begin{ttt}[\cite{Yom}]\label{lb8}{\sl
Let $m\in\{0,\dots,n-1\}$ and  $m<d$.  Then for any polynomial $P\colon \R^n\to\R^d$ of degree at most $k$,
for each $n$-dimensional  cube $Q\subset \R^n$ of size
$\ell(Q)=r>0$, and for any number $\varepsilon >0$ we have that
$$
\E\bigl(\varepsilon r,\{P(x):x\in Q,\ \lambda_1
\le1+\varepsilon,\dots,\lambda_{m } \le1+\varepsilon, \lambda_{m+1}
\le\varepsilon,\dots,\lambda_d \le\e\}\bigr)
$$
$$\le C_Y \bigl(1+\e^{-m}\bigr),
$$
where the~constant $C_Y$ depends on $n,d,k,m$ only and for brevity
we wrote $\lambda_{j}= \lambda_{j}(P,x)$.}
\end{ttt}

\subsection{On Choquet type integrals}

Recall the following classical theorem referred to
D.R.~Adams, see, e.g., \cite{Ad1}--\cite{Ad2} or \cite{Ad3}.

\bigskip

\begin{ttt}\label{AM1}{\sl
Let $\beta>0$, \,$n-\beta p>0$, \,and \,$s>p>1$. Then for any $g\in
L_p(\R^n)$ the estimate
\begin{equation}
\label{Ri1}
\int_0^\infty\H^{\tau}_\infty( \{x\in\R^n : \M \bigl(I_\beta g\bigr)(x)\ge t^{\frac1s} \})\,\dd t \le C\Vert g\Vert^s_{\LL_{p}}
\end{equation}
holds with $\tau=\frac{s}p(n-\beta p)$, where $C$ depends on $n,
\ p, \ s, \  \beta$ only. }
\end{ttt}
Here $$I_{\beta}g(x):=\int_{\R^n}\frac{g(y)}{|y-x|^{n-
{\beta}}}\,\dd y$$ is  the classical  Riesz potential of order $\beta$, and
$$
\M f(x)=\sup\limits_{r>0}  \dashint_{B(x,r)} \! |f(y)|\,\dd y
$$
is the usual Hardy--Littlewood maximal function of~$f$.
\bigskip

The above estimate (\ref{Ri1}) fails for the limiting case $s=p$.
Namely, there exist functions $g\in L_p(\R^n)$ such that
$| I_\beta g|(x)=+\infty$ on some set of positive
$(n-\beta p)$--Hausdorff measure. One possible way to cover this limiting case~$s=p$ is  using the Lorentz norm instead of Lebesgue norm in the right hand side of~(\ref{Ri1}).  Such possibility was proved in the recent paper~\cite{KK15}.

\begin{ttt}[see Theorem~0.2 in~\cite{KK15}]\label{lb7}{\sl
Let $\beta>0$, \,$n-\beta p>0$, \,and \,$p>1$. Then for any $g\in
L_p(\R^n)$ the estimate
\begin{equation}
\label{Ri5}
\int_0^\infty\H^{\tau}_\infty( \{x\in\R^n : \M \bigl(I_\beta g\bigr)(x)\ge t^{\frac1p} \})\,\dd t \le C\Vert g\Vert^p_{\LL_{p,1}}
\end{equation}
holds with $\tau=n-\beta p$, where $C$ depends on $n,
\ p, \  \beta$ only. }
\end{ttt}

\bigskip

The above theorems are not fulfilled in general for $p=1$. However, similar results hold in case  $p=1$ for derivatives of Sobolev mappings. Namely, the following Theorem was proved by D.R.~Adams~\cite{Ad2}.

\begin{ttt}\label{cs-lb7}{\sl
Let $k,l\in\{1,\dots,n\}$, $l<k$. Then for any function $f$ from the Sobolev space $\WW^{k}_{1}(\R^n)$ the
estimates
\begin{equation}
\label{adcb3}
\int_0^\infty\H^{\tau}_\infty( \{x\in\R^n : \M \bigl(\nabla^l
f\bigr)(x)\ge t \})\,\dd t \le C\Vert \nabla^k
f\Vert_{\LL_{1}}
\end{equation}
hold, where $\tau=n-k+l$ and the constant $C$ depends on $n,k,l$. }
\end{ttt}

The application of above estimates on maximal functions  is facilitated through the
following simple Lipschitz type  inequality  (see for instance Lemma~2 in \cite{Dor},
cf. with~\cite{BH}\,).

\begin{lem}\label{lb10} {\sl Let $u\in\WW^{1}_{1}(\R^n,\R^d)$. Then for
any ball $B\subset \R^n$ of radius $r>0$ and for any number
$\varepsilon >0$ the estimate
$$
\diam (\{u(x):x\in B,\ (\M \nabla u)(x)\le\varepsilon\})\le C_M
\varepsilon r
$$
holds, where $C_{M}$ is a constant depending on $n,d$ only.}
\end{lem}

Using the similar calculations, one could obtain the following refinement of the above Lemma. 

\begin{lem}\label{l-lip} {\sl Let $u\in\WW^{1}_{1}(Q,\R^d)$, where $Q$ is an~$n$-dimensional interval. Then for
any ball $B\subset \R^n$ of radius $r>0$ and for any number
$\varepsilon >0$ the estimate
$$
\diam (\{u(x):x\in B\cap Q,\ (\M_Q \nabla u)(x)\le\varepsilon\})\le C_M
\varepsilon r
$$
holds, where $C_{M}$ is a constant depending on $n,d$ only, and 
$$\M_Qf:=\M(1_Q\cdot f),$$
i.e., \begin{equation}\label{lip-m}
\M_Qf(x)=\sup\limits_{r>0}  \frac1{|B(x,r)|}\int_{Q\cap B(x,r)} \! |f(y)|\,\dd y.\end{equation}}
\end{lem}

\subsection{On Fubini type theorems  for graphs of continuous functions}

Recall that by usual Fubini theorem, if a set $E\subset\R^2$ has a zero plane measure, then for $\H^1$-almost all straight lines $L$ parallel to coordinate axes we have $\H^1(L\cap E)=0$. The next
result could be considered as functional Fubini type theorem.

\begin{ttt}[see Theorem~5.3 in~\cite{HKK}]\label{FubN}{\sl
Let $\mu\ge 0$, \,$q>0$, \,and \,$v:\R^n\to\R^d$ \,be a continuous function. For a set $E\subset\R^n$
define the set function
\begin{equation*}\label{dd5}
\Phi(E)=\inf\limits_{E\subset\bigcup_j
D_j}\sum\limits_j\bigl(\diam D_j\bigr)^\mu\bigl[\diam
v(D_j)\bigr]^q,
\end{equation*}
where the infimum is taken over all countable families of compact
sets $\{D_j\}_{j\in \N}$ such that $E\subset\bigcup_j D_j$. Then \ $\Phi(\cdot)$ is a countably
subadditive and the implication
\begin{equation*}\label{dd6} \Phi(E)=0\ \boldsymbol{\Rightarrow}\
\biggl[\H^\mu\bigl(E\cap v^{-1}(y)\bigr)=0\quad\mbox{for
$\H^q$-almost all }y\in\R^d\biggr]
\end{equation*}holds.}
\end{ttt}

\subsection{On local properties of considered potential spaces}\label{loc-s1}
Let $\Be$ be some space of functions defined on~$\R^n$. For a set $\Omega\subset\R^n$ define the space $\Be_\loc(\Omega)$ in the following standard way:
$$\Be_\loc(\Omega):=\{f:\Omega\to\R:\mbox{ for any compact set }E\subset\Omega\mbox{ $\exists g\in\Be$ such that }f(x)=g(x)\ \forall x\in E\,\}.$$
Put for simplicity $\Be_\loc=\Be_\loc(\R^n)$.

It is well known that for $q>p>1$ the inclusions
\begin{equation*}\label{loc00}L_{q,\loc}\subset L_{p,1,\loc}\subset L_{p,\loc},
\end{equation*}
hold (see, e.g., \cite{Maly2}\,). Respectively, it is easy to see that for $\al>0$ one has
\begin{equation*}\label{loc1} \Le^\al_{q,\loc}\subset \Le^\al_{p,1,\loc}\subset \Le^\al_{p,\loc}.
\end{equation*}
Since the Morse--Sard type theorems have a local nature, if we prove some of these theorems for $\Le^\al_{p}$, then the same result will be valid for
the spaces $\Le^\al_{p,1}$ and $\Le^\al_{q}$ for all $q>p$. Similarly, if we prove some Morse--Sard type theorems for $\Le^\al_{p,1}$, then the same result will be valid for
the spaces $\Le^\al_{q}$ with $q>p$, etc.

\subsection{Approximation by Holder--smooth functions}\label{app-s1}

We need also the following approximation result.

\begin{ttt}[see, e.g., Chapter~3 in \cite{Ziem} or
\cite{BHS1}]\label{Th_ap}{\sl Let $p>1$, $k\in\N$,
\,$\alpha\in(0,1)$. Then for any $f\in\Le^{k+\alpha}_p(\R^n)$ and
for each $\varepsilon>0$ there exist an open set $U\subset\R^n$
and a function $h\in \CC^{k,\alpha}(\R^n)$ such that

\begin{itemize}
\item[(i)] \,$\Lb^n(U)<\varepsilon$;

\item[(ii)] \,each point $x\in\R^n\setminus U$ is {an~Lebesgue
point for~$f$ and $\nabla f$;}

\item[(iii)] \,$f\equiv h$ and $\nabla f\equiv\nabla h$ on $\R^n\setminus U$.
\end{itemize}
}
\end{ttt}

Note, that in the cited references the approximation property is
discussed for the case of Sobolev spaces $W^k_p$, but the proof
for the $\Le^{k+\alpha}_p(\R^n)$ space easily follows from the just mentioned
Sobolev case and some standard arguments on real analysis
concerning approximation limits and Whitney-type extension
theorems for Holder classes (see, e.g., Theorem~4 in \cite[\S2.3,
Chapter~6]{St}\,).

\section{Proofs of the main results}

\subsection{Bridge Federer--Dubovitski\u{\i} theorem for Sobolev mappings}\label{MSP}

Recall, that \textit{bridge
Dubovitski\u{\i}--Federer Theorem} Theorem~\ref{DFT-F} for the case~(ii) was proved in our previous paper~\cite{HKK}. The purpose here is to prove Theorem~\ref{DFT-F}~(i).
But of course, the case~(i) with $k\le n$ follows immediately from the case~(ii) (see the section~\ref{loc-s1}\,). So we need to consider here only the situation Theorem~\ref{DFT-F}~(i)
with $$k>n\quad\mbox{ and }\quad p=1.$$

Fix integers $m\in\{0,\dots,n-1\}$, $d> m$, $k>n$, and a mapping
$v\in \WW^{k}_{1}(\R^n,\R^d)$. Then, by Lemma~\ref{lb3} the
function~$v$ is $C^1$-smooth.

Denote $Z_{v,m}=\{x\in\R^n:\rank\nabla v(x)\le m\}$. Fix a~number
$q>m.$ Denote in this subsection
\begin{equation*}\label{muu3}
\mu=\mu_q=n-m-k(q-m).
\end{equation*}

Recall, that we need  to consider only  the case 
\begin{equation*}\label{qqq}
q\in(m,\b],
\end{equation*}
where $\b =m+\frac{n-m}k$ (see Remark~\ref{remq0}\,). Then by direct calculation we have
$\mu\ge0$.

The required assertion of the \textit{bridge
Dubovitski\u{\i}--Federer Theorem}~\ref{DFT-F}-(i) is equivalent (by
virtue of Theorem~\ref{FubN}\,) to the identity
\begin{equation*}\label{cor-evv1}
\Phi(Z_{v,m})=0,
\end{equation*}
where by definition \begin{equation}\label{h-dd5-dd}
\Phi(E):=\inf\limits_{E\subset\bigcup_j
D_j}\sum\limits_j\bigl(\diam D_j\bigr)^\mu\bigl[\diam
v(D_j)\bigr]^q.
\end{equation}
As indicated the infimum is taken over all countable families of
compact sets $\{D_j\}_{j\in \N}$ such that $E\subset\bigcup_j
D_j$.

Before embarking on the detailed proof we make some preliminary
observations that allow us to make a few simplifying assumptions.
We could assume without loss of generality that
\begin{equation*}\label{ass8}
|\nabla v(x)|\le1\qquad\forall x\in\R^n.
\end{equation*}
\noindent Denote $Z_v=Z_{v,m}$.  Then from Theorem~\ref{Sl-est-4}~(i), applied to our values of $q,\mu=\mu_q$ we obtain immediately

\begin{lem}\label{lb11}{\sl
Let $q\in(m,\b]$. Then for any sufficiently small
$n$-dimensional interval $Q\subset\R^n$ the estimate
\begin{equation}
\label{6} \Phi(Z_v\cap Q)\leq C\,\ell(Q)^{n(m-q+1)}\cdot\|\nabla^{k}
v\|^{q-m}_{\LL_1(Q)}
\end{equation}
holds, where the constant $C$ depends on $n,m,k,d$ only.}
\end{lem}
(Indeed, we have by direct elementary calculation, that now the exponent $q+\mu+(k-1-n)(q-m)$ from the formula~(\ref{t-cs-2}) coincides with $n(m-q+1)$ from~(\ref{6}).)

Note, that by our assumptions $k>n$, therefore
\begin{equation*}
\label{qqe} q\le\b=m+\frac{n-m}k<m+1.
\end{equation*}

\begin{cor}
\label{cor00} {\sl  Let $q\in(m,\b]$. Then for any
$\varepsilon>0$ there exists $\delta>0$ such that for any subset
$E$ of $\R^n$ we have $\Phi(Z_v\cap E)\le\varepsilon$ provided
$\Lb^n(E)\le\delta$. In particular, $\Phi(Z_{v}\cap E)=0$ whenever
$\Lb^n(E)=0$.}
\end{cor}

\begin{proof}
 Let $\Lb^n(E)<\delta$. Then we can find a family of
nonoverlapping $n$-dimensional dyadic intervals~$Q_j$ such that
$E\subset\bigcup_j Q_j$ and $\sum\limits_j\ell^n(Q_j)<C\delta$. Of
course, for sufficiently small $\delta$ the estimates
\begin{equation*}\label{eles1}\|\nabla^{k}
v\|_{\LL_1(Q_j)}<1,\qquad \ell(Q_j)\le \delta^\frac1n
\end{equation*} are fulfilled for every~$j$. Denote
\begin{equation}\label{ds5}r_j=\ell(Q_j),\qquad\sigma_j=\|\nabla^{k}
v\|_{\LL_1(Q_j)},\qquad\sigma=\|\nabla^{k}
v\|_{\LL_1}.\end{equation} In view of Lemma~\ref{lb11} we have
\begin{equation*}\label{oxMS1}\Phi(E)\le C\sum_j
r_j^{n(m-q+1)}\,\sigma_j^{q-m}.
\end{equation*}
Since by our assumptions
$$
0<m-\b+1\le m-q+1<1,\qquad
$$
we have
\begin{eqnarray*}\label{oxMS3}
\sum_j r_j^{n(m-q+1)}\,\sigma_j^{q-m}
&\overset{\mbox{\footnotesize\color{red}H\"{o}lder ineq.}}\leq &
C \biggl(\sum\limits_j r_{j}^{n}\biggr)^{m-q+1}\cdot \biggl(\sum\limits_j\sigma_j\biggr)^{q-m} \nonumber\\
&\leq& C' \delta^{m-q+1}\cdot\sigma^{q-m} .
\end{eqnarray*}
The lemma is proved.
\end{proof}
By the classical approximation results (see, e.g., Chapter~3 in \cite{Ziem} or
\cite{BHS1}\,), our mapping~$v$ coincides with a
mapping $g\in \CC^{k}(\R^n,\R^d)$ off an exceptional set of small
$n$--dimensional Lebesgue measure. So we need to check the~assertion of Theorem~\ref{DFT-F}-(i)  for
$\CC^k$--smooth mappings now.

\begin{lem}
\label{Dub-smooth} {\sl  Let $q\in(m,\b]$ \ and \ $g\in
\CC^{k}(\R^n,\R^d)$, $k>n$. Then
\begin{equation}
\label{ff-ds} \Phi_g(Z_{g,m})=0,
\end{equation} where $\Phi_g$ is calculated by the same
formula~(\ref{h-dd5-dd}) with $g$ instead of~$v$ and
$Z_{g,m}=\{x\in\R^n:\rank\nabla g(x)\le m\}$. }
\end{lem}

\begin{proof}
We can assume without loss of generality that $g$ has compact
support and that $| \nabla g(x)| \leq 1$ for all $x \in \R^n$. We
then clearly have that $g\in \WW^{k}_{1}(\R^n,\R^d)$, hence we can
in particular apply the above results to $g$. The following
assertion plays the key role:

\noindent ($*$) \ {\sl For any $n$-dimensional  interval
$Q\subset\R^n$ the estimate
\begin{equation*}
\label{ff-ds-kk} \Phi(Z_{g,m}\cap Q)\le
C\,\ell(Q)^{n(m-q+1)}\,\|\nabla^{k} \bg_Q\|^{q-m}_{\LL_1(Q)}
\end{equation*}
holds, where the constant $C$ depends on $n,m,k,d$ only, and we
denoted
$$
\nabla^{k}\bg_Q(x)=\nabla^{k}g(x)-\dashint\limits_Q\nabla^kg(y)\,
\dd y.
$$}
The proof of ($*$) is almost the same as that of Lemma~\ref{lb11},
with evident modifications (we need to take the approximation
polynomial~$P_Q(x)$ of degree $k$ instead of~$k-1$, etc.).

By elementary facts of the Lebesgue integration theory, for an
arbitrary family of nonoverlapping $n$-dimensional 
intervals~$Q_j$ one has
\begin{equation}
\label{6-ds}\sum\limits_j \|\nabla^{k}
\bg_{Q_j}\|_{\LL_1(Q_j)}\to0\qquad\mbox{ as \
}\sup\limits_j\ell(Q_j)\to 0
\end{equation}
The proof of this estimate is really elementary since now
$\nabla^kg$ is a~continuous and compactly supported function, and,
consequently, is uniformly continuous and bounded.

From ($\ast$) and (\ref{6-ds}), repeating the arguments of
Corollary~\ref{cor00}, using the assumptions on $g$ and taking
$$
\sigma_j=\|\nabla^{k} \bg_{Q_j}\|_{\LL_1(Q_j)},\qquad
\sigma=\sum\limits_j \sigma_j
$$
in definitions~(\ref{ds5}), we obtain that $\Phi_g(Z_{g,m})<\e$
for any $\e>0$, hence the sought conclusion~(\ref{ff-ds}) follows.
\end{proof}

By the above--mentioned approximation results, the~investigated mapping~$v$ equals a~mapping $g\in
\CC^{k}(\R^n,\R^d)$ off an exceptional set of small
$n$--dimensional Lebesgue measure. This fact together with
Lemma~\ref{Dub-smooth} readily implies

\begin{cor}[cp.~with \cite{DeP}]
\label{Th_ap2} {\sl  Let $q\in(m,\b]$. Then there exists a set
$\widetilde Z_{v}\subset Z_v$ of $n$-dimensional Lebesgue measure zero such
that $\Phi(Z_v\setminus \widetilde Z_{v})=0$. In particular,
$\Phi(Z_v)=\Phi(\widetilde Z_{v})$.}
\end{cor}
From Corollaries~\ref{cor00} and \ref{Th_ap2} we conclude that
$\Phi(Z_v)=0$, and this finishes the proof of Theorem~\ref{DFT-F}-(i) for the required case~$k>n$, $p=1$.

\medskip

\begin{rem}\label{BBBV} As we could see from the above proofs, the assertion of Theorem~\ref{DFT-F}-(i)  is valid also under assumption $v\in\BV_k(\R^n,\R^d)$ \,(instead of $W^k_1$\,) \,with the
same~$k$. Here $BV_k$ means the space of functions $v\in
W^{k-1}_1$ such that its $k$-th (distributional) derivatives
are~Radon measures.
\end{rem}

\begin{rem}\label{SHH1} Thus, the assertions of Theorem~\ref{DFT-F}-(i)-(ii) are proved. Of course, these assertions imply 
the~fulfillment of the~statement of Theorem~\ref{DST-H} 
for the extreme borderline  cases $\alpha=0$ and $\alpha=1$, because the corresponding Holder spaces $C^{k}$ and $C^{k,1}$ are contained in the  Sobolev spaces 
$W^k_p$ and $W^{k+1}_p$ (with $p>n$) respectively. 
\end{rem}

\subsection{ Bridge F.-D. Theorem for Holder classes of mappings}\label{HT}

This subsection is devoted to the proof of Theorem~\ref{DST-H}. It is sufficient to consider the general case~$0<\al<1$ (see Remark~\ref{SHH1}).

\medskip

Fix $m\in\{0,\dots,n-1\}$, \,$k\ge1$, \,$d>m$, \,$0<\al<1$, \,and
\,$v\in \Cc(\R^n,\R^d)$. Take also a~parameter $q>m$. Of course,
as in the~previous subsection, it is sufficient to consider the case
\begin{equation*}\label{cor-q8} q\in(m,\b],
\end{equation*}
where $\b =m+\frac{n-m}{k+\alpha}$. By definition of the space
$\Cc$ and since the result has the~local nature, we may assume without
loss of generality that
\begin{eqnarray}\label{hhc1} &&\mbox{$|\nabla^kv(x)-\nabla^kv(y)|\le |x-y|^\alpha$ \ for all $x,y\in\R^n$.}\\
\label{hhc2} &&|\nabla v(x)|\le 1\qquad \mbox{for all $x\in\R^n$,}
\\
\label{hhc3} &&v(x)\equiv0\qquad \mbox{if \ $|x|>1$.}
\end{eqnarray}

Denote again $Z_v=Z_{v,m}=\{x\in\R^n:\rank\nabla v(x)\le m\}$. Now
the parameter $\mu$ is different from the previous subsection: 
\begin{equation}\label{hk7}
\mu=n-m-(k+\al)(q-m).
\end{equation}
As before, the assertion of the Bridge Dubovitski\u{\i}--Federer
Theorem~\ref{DST-H} is equivalent (by virtue of Theorem~\ref{FubN}\,) to
\begin{equation*}\label{cor-ev1}
\Phi(Z_{v})=0,
\end{equation*}
where we denoted
\begin{equation*}\label{cor-jjj}
\Phi(E)=\inf\limits_{E\subset\bigcup_j
D_j}\sum\limits_j\bigl(\diam D_j\bigr)^\mu\bigl[\diam
v(D_j)\bigr]^q.
\end{equation*}
As indicated the infimum is taken over all countable families of
compact sets $\{D_j\}_{j\in \N}$ such that $E\subset\bigcup_j
D_j$.

The proof of Theorem~\ref{DST-H} consists of several steps.

\

{\sc Step I}.  Applying Theorem~\ref{Hold-est-1} to the present case $\mu$ defined by~(\ref{hk7}), we obtain immediately the following
assertion, which is  is the main tool for further arguments:

\medskip

({\color{red}$**$}) {\sl Under above assumptions on~$v$,   for an
arbitrary sufficiently small $n$-dimensional cube~$Q$ of size $e=\ell(Q)$ the
estimate
\begin{equation}
\label{hst5} \Phi(Z_{v}\cap Q)\le C\,r^n\,
\end{equation}
holds, where the constant $C$ depends on $n,m,k,d,\al$ only. }

\medskip

Of course, the last estimate implies
\begin{equation}\label{f-nad1}\Phi(Z_{v}\cap F)\le \Psi(Z_{v}\cap F)\le C\cdot\Lb^n(F)
\end{equation}
for any measurable set~$F\subset \R^n$, where  the countably subadditive set function $\Psi$ is defined as 
\begin{equation}\label{cs-finit-mdd5}
\Psi(F)=\lim\limits_{\delta\to0}\inf\limits_{\footnotesize{\begin{array}{lcr}F\subset\bigcup_j
D_{j},\\
\diam D_{j}\le\delta\end{array}}}\sum\limits_j\bigl(\diam
D_{j}\bigr)^{\mu}\bigl[\diam v(D_{j})\bigr]^q.
\end{equation}
Here the infimum is taken over all countable families of
compact sets $\{D_{j}\}_{{j}\in \N}$ such that
$F\subset\bigcup_{j} D_{j}$ and $\diam D_{j}\le\delta$ for
all~${j}$.

\

{\sc Step II: the case $m=0$}.

Suppose now that~$m=0$.
In other words, now $Z_{v}=\{x\in\R^n:\nabla v(x)=0\}$. Thus
$\nabla^kv(x)\equiv0$ for almost all $x\in Z_v$. Then we have the
decomposition:
\begin{equation*}
\label{nc2}Z_{v,m}=E_0\cup E_1,
\end{equation*}
where $\Lb^n(E_0)=0$, and every $x\in E_1$ is a density point for the set
$\{x\in\R^n:\nabla^kv(x)=0\}$. It implies, by elementary arguments,
that
\begin{equation*}
\label{nc3}\lim\limits_{y\to
x}\frac{|\nabla^kv(y)-\nabla^kv(x)|}{|x-y|^\alpha}\to
0\qquad\forall x\in E_1.
\end{equation*}
Then, checking the proof of the basic estimate~(\ref{hst5}) (see also~(\ref{app-6-new})\,), we
see that for any point $x\in E_1$ the identity
\begin{equation}
\label{nc3-lll}\lim\limits_{r\to0}\frac{\Phi(Z_v\cap Q(x,r))}{r^n}=0\end{equation}
holds, where $Q(x,r)$ denotes the cube centered at~$x$ with $\ell(Q)=r$.
By usual elementary facts of real analysis and by subadditivity of
$\Phi(\cdot)$, the~convergence~(\ref{nc3-lll})
\  implies that $\Phi(E_1)=0$. The equality
$\Phi(E_0)=0$ follows from the condition $\Lb^n(E_0)=0$
and~(\ref{f-nad1}). So $\Phi(Z_v)=0$ as required. The case $m=1$ is
finished completely.

\

{\sc Step III: the case $m>0$}.

From this point, for all the steps below we assume that $m\ge 1$.
 By definitions, we have 
\begin{equation}
\label{cor-c0} Z_{v}=E_{0}\cup E_{1},
\end{equation} where 
$$E_0=\{x\in\R^n:\rank\nabla v(x)<m \}, \qquad E_{m }:=\{x\in\R^n:\rank\nabla v(x)=m \}.$$
By
construction, $E_0\subset Z_{v,m-1}$, so we could apply the
previous estimate (\ref{f-nad1}) for $m'=m-1$ instead of~$m$
to obtain \begin{equation} \label{hst11}
\widetilde\Psi(E_0)<\infty,
\end{equation}
where $\widetilde\Psi$ is defined as~$\Psi$ (see
(\ref{cs-finit-mdd5})\,) with $\mu$ replaced by
$$\tilde\mu=n-m' -(k+\alpha)(q-m' )= \mu -(k+\alpha-1)<\mu.$$
Of course, the inequalities $\widetilde\Psi(E_0)<\infty$ \, and \,$\tilde\mu<\mu$ imply $\widetilde\Psi(E_0)\ge
\Psi(E_0)=0$, and, consequently, 
\begin{equation}
\label{cor-c1} \Phi(E_0)=0.
\end{equation}

\

{\sc Step IV}.  Now we have to estimate the last term 
$\Phi(E_{1})$ \ with \ $\rank\nabla v|_{E_{1}}\equiv m$. 
By Implicit Function Theorem and by the local nature of the considered results, we can assume without loss of generality, that 
\begin{equation}
\label{cor-c3}v(x)= v(y,z)=(y,f(y,z))\qquad\forall x\in E_{1},
\end{equation}
where for $x=(x_1,\dots,x_n)\in\R^n$ we denote 
$y=(x_1,\dots,x_{m})$, \ $z=(x_{m+1},\dots, x_n)$, in particular, 
$x=(y,z)$, and $f:\R^{n-m}\to\R^d$ is some $C^{k,\al}$-Holder mapping. 

Now the condition  $\rank\nabla v|_{E_{1}}\equiv m$  can be rewritten in the following equivalent form:
\begin{equation}
\label{cor-c4}\nabla_zv(y,z)\equiv 0\qquad\forall (y,z)\in E_{1},
\end{equation}
where we denote $\nabla_zv=\frac{\partial v}{\partial z}=\bigl(\frac{\partial v}{\partial x_{m+1}},\dots,\frac{\partial v}{\partial x_{n}}\bigr)$.

From formula~(\ref{cor-c4}) using Fubini's Theorem and standard facts about density points of Lebesgue measurable sets it is easy to deduce that 
there exists a~decomposition $E_1=E_{*}\cup E_{**}$ such that 
\begin{equation}
\label{cor-c5}\Lb^n(E_{*})=0,
\end{equation}
\begin{equation}
\label{cor-c6}\nabla^l_zv(y,z)\equiv 0\qquad\forall (y,z)\in E_{**}\ \ \forall l=1,\dots,k; 
\end{equation}
\begin{equation}
\label{cor-c7}\lim\limits_{z\to z_0}\frac{|v(y_0,z)-v(y_0,z_0)|}{|z-z_0|^{k+\al}}=0\qquad\forall (y_0,z_0)\in E_{**}.
\end{equation}
Then from Egoroff's Theorem on uniform convergence of measurable functions, for any $\delta>0$  there exist a decomposition 
\begin{equation}
\label{cor-c8}E_{**}=F\cup E
\end{equation}
 such that 
 \begin{equation}
\label{cor-c9}\Lb^n(F)<\delta,
\end{equation}
and the {\bf uniform } convergence 
\begin{equation}
\label{cor-c10}\frac{|v(y, z+h)-v(y,z)|}{|h|^{k+\al}}\underset{h\to 0}\to 0\qquad \forall (y,z)\in E
\end{equation}
holds. Here   $0\ne h\in \R^{n-m}$, and convergence is uniform with respect to~$h$, i.e., for any $\xi>0$ there exists $\rho_\xi>0$ such that 
for every $ (y,z)\in E $ \  and for each $h\in \R^{n-m}$ satisfying $0<|h|<\rho_\xi$ the inequality 
\begin{equation}
\label{cor-c11}|v(y, z+h)-v(y,z)|<\xi\,|h|^{k+\al}
\end{equation}
holds. 

\

{\sc Step V}. We claim that for any $x_0=(y_0,z_0)\in E $ the convergence
\begin{equation}
\label{cor-c11-f}\lim\limits_{r\to0}\frac{\Phi(E\cap Q(x_0,r))}{r^n}=0
\end{equation}
holds, where we denote by~$Q(x,r)$ the $n$-dimensional cube centered at~$x$ with side length~$r$. 

This claim is proved by the direct elementary calculations. Indeed, fix arbitrary $\xi\in(0,1)$ and take~$\rho_\xi<1$ from the previous Step~IV. Let $0<r<\rho_\xi$. The cube
$Q=Q(x_0,r)$ can be represented as $Q=Q_y\times Q_{z}$, where $Q_y$ and $Q_z$ are the corresponding cubes of the same size in spaces~$\R^{m}$ and $\R^{n-m}$ respectively.

Denote $\e=\xi\,r^{k+\al-1}$. Assume without loss of generality that~$\frac1\e$  is an integer number, put
\begin{equation}
\label{cor-c15}N=\e^{-m}, 
\end{equation}
and consider the decomposition 
$$Q_y=\bigcup\limits_{j=1}^N Q_j,$$
where every $Q_j$ is the corresponding cube of size $\e r$ in the~''$y$-space''\ \ $\R^{m}$. Denote further
$$J=\{j=1,\dots,N: E \cap (Q_j\times Q_z)\ne\emptyset\}.$$

By construction, every rectangle $ Q_j\times Q_z$, $j\in J$, has the size $\e r$ and $1$-Lipschitz condition (see~(\ref{hhc2})\,) along $y$-coordinates, and respectively size $r$ and 
the strong Lipschitz conditions~(\ref{cor-c11}) along $z$-coordinates, that  imply
\begin{equation}
\label{cor-c13}\diam v(Q_j\times Q_z)\le 2\big(\e r+\xi r^{k+\al}\bigr)=4\e r=4\xi r^{k+\al}. 
\end{equation}
Finally we have 
\begin{equation}
\label{cor-c14}n^{-\frac\mu2}\Phi(E\cap Q)\le \sum\limits_{j\in J}r^\mu\bigl[\diam v(Q_j\times Q_z)\bigr]^q\le(4\xi)^q N r^{(k+\al)q+\mu}\overset{\footnotesize(\ref{cor-c15}),\,(\ref{hk7})}=
4^q\xi^{q-m}r^n. 
\end{equation}
Since $\xi$ could be taken arbitrary small, the proof of the Claim~(\ref{cor-c11-f}) is complete. Of course, this claim and the countable subadditivity of~$\Phi$  imply immediately 
\begin{equation}
\label{cor-c17}\Phi(E)=0.
\end{equation}

\medskip

Let us summarize what were done on the previous steps  III--V. For arbitrary $\delta>0$ we obtain the decomposition
$$Z_v=E_0\cup E_*\cup F\cup E,$$
where $\Phi(E_0)=\Phi(E)=0$, $\Lb(E_*)=0$, and $\Lb(F)<\delta$. 
By virtue of  estimate~(\ref{f-nad1}) this imply 
$$\Phi(Z_v)< C\,\delta,$$
consequently, 
$$\Phi(Z_v)=0.$$
Thus  Theorem~\ref{DST-H} is proved completely.

\subsection{Bridge F.-D. Theorem for mappings of  potential spaces~$\Le^{k+\al}_p(\R^n,\R^d)$: case $k\le n$.}\label{HTFSF}
 This subsection is devoted to the proof of Theorem~\ref{DFT-F}, case~(iii). Here we consider the situation
when $1\le k\le n$ and $0<\al<1$. Then $k+\al-1<n$.  Since our results have a local nature  (see Subsection~\ref{loc-s1} for more precise explanations),
it is sufficient to consider only the case
\begin{equation}\label{dim1}n-p<(k+\al-1)p<n
\end{equation}
i.e., when $v$ is  a continuous function by Sobolev Imbedding Theorems, but the gradient $\nabla v$ could be discontinuous.

Fix $m\in\{0,\dots,n-1\}$, \,$d>m$, \,and
\,$v\in\Le^{k+\al}_{p}(\R^n,\R^d)$. In other words,
$$v=\GG_{k+\alpha}(g):=G_{k+\al}*g$$ fore some $g\in L_{p}(\R^n)$.

Take a~parameter $q>m$. Of course,
as in previous subsection, it is sufficient to consider the case
\begin{equation*}\label{lll-frac-cor-q8} q\in(m,\b],
\end{equation*}
where $\b =m+\frac{n-m}{k+\alpha}$. In particular, we have
\begin{equation}\label{csl00}q-m\le \b-m=\frac{n-m}{k+\al}< p.
\end{equation}
Now the parameter $\mu$ is the same as in the previous subsection:
\begin{equation}\label{mmmm1}
\mu=n-m-(k+\al)(q-m).
\end{equation}

Denote again $Z_{v,m}=\{x\in\R^n\setminus A_v:\rank\nabla
v(x)\le m\}$. Here $A_v$ means the set of `bad' points, where $v$ is not  differentiable or
$\M\nabla v=\infty$. Recall that this set has a small size, namely,
\begin{equation*}\label{lll-hf0-ii}\H^\tau(A_v)=0\qquad\ \forall\tau>\tt=n-(k+\al-1)p.
\end{equation*}

As before, the assertion of the Bridge Dubovitski\u{\i}--Federer
Theorem~\ref{DFT-F} is equivalent (by virtue of
Theorem~\ref{FubN}\,) to
\begin{equation*}\label{lll-frac-cor-ev1}
\Phi(Z_{v,m})=0,
\end{equation*}
where
\begin{equation*}\label{lll-frac-cor-jjj}
\Phi(E):=\inf\limits_{E\subset\bigcup_j
D_j}\sum\limits_j\bigl(\diam D_j\bigr)^\mu\bigl[\diam
v(D_j)\bigr]^q.
\end{equation*}
Since $\Phi$ is a~subadditive set-function, it is sufficient to check only the equality
\begin{equation*}\label{lll-frac-cor-ev1-iii}
\Phi(Z_{v})=0,
\end{equation*}
where again we denote
\begin{equation*}\label{lll-frac-cor-jjj-iii}
Z_v:=\{x\in Z_{v,m}:|\nabla v(x)|\le 1\}.
\end{equation*}
From the inequality  $q-m< p$ \  (see (\ref{csl00})\,) we conclude that
\begin{equation}\label{csl1} q+\mu\overset{\footnotesize(\ref{mmmm1})}=n-(k+\al-1)(q-m)> n-(k+\al-1)p=\tt,
\end{equation}
so the key assumption~(\ref{s-as}) of Theorem~\ref{Sob-est-2} is fulfilled now.
\medskip

{\sc Step I: estimates on a single cube}.  Applying the just mentioned  Theorem~\ref{Sob-est-2}
 to the present case $\mu=\mu_q=n-m-(k+\al)(q-m)$, we obtain immediately  the following
assertion, which is  is the main tool for further arguments.
\medskip

\begin{lem}\label{l-sp1}{\sl
 Under above assumptions on~$v$,   for an~arbitrary  $n$-dimensional cube~$Q$ of size $r=\ell(Q)$ the
estimate
\begin{equation}
\label{sp1} \Phi(Z_{v}\cap Q)\le C\,\biggl(\,\sigma^q r^{(k+\al-\frac{n}p)q+\mu}\ +\ \sigma^{q-m}r^{n\bigl(1-\frac{q-m}p\bigr)}\biggr),
\end{equation}
holds, where \begin{equation}
\label{sp2}\sigma=\|\M g\|_{L_{p}(Q)},
\end{equation}
and  the constant $C$ depends on $n,m,k,\alpha,d,p$ only.}
\end{lem}

\

{\sc Step II: estimates on sets of small $n$-Lebesgue measure.}

\begin{lem}
\label{cor00-sl} {\sl Under above assumptions on~$v$, for any
$\varepsilon>0$ there exists $\delta>0$ such that for any subset
$E$ of $\R^n$ we have $\Phi(Z_v\cap E)\le\varepsilon$ provided
$\Lb^n(E)\le\delta$. In particular, $\Phi(Z_{v}\cap E)=0$
whenever $\Lb^n(E)=0$.}
\end{lem}

\begin{proof}
 Let $\Lb^n(E)\le\delta$, then we can find a family of
nonoverlapping $n$-dimensional dyadic intervals~$Q_\beta$ such
that $E\subset\bigcup_\beta Q_\beta$ and
$\sum\limits_\beta\ell^n(Q_\beta)<C\delta$. Of course, for
sufficiently small $\delta$ the estimates
\begin{equation}\label{eles1}\|\nabla^{k}
v\|_{\LL_{p}(Q_\beta)}<1,\qquad \ell(Q_\beta)\le
\delta^\frac1n<1
\end{equation} are fulfilled for every~$\beta$. Denote
\begin{equation}\label{fs-ds5}r_\beta=\ell(Q_\beta),\qquad\sigma_\beta=\|\nabla^{k}
v\|_{\LL_{p}(Q_\beta)},\qquad\sigma=\|\nabla^{k}
v\|_{\LL_{p}(\cup_\beta Q_\beta)}.\end{equation} In view of Lemma~\ref{l-sp1} we
have
\begin{equation}\label{oxMS1}\Phi(E)\le C\sum_\beta
\sigma_\beta^{q-m}r_\beta^{{(1-\frac{n}p)(q-m)}}+C\sum_\beta
 \sigma_\beta^qr_\beta^{(k+\al-\frac{n}p)q+\mu}.
\end{equation}
Now let us estimate the first sum. Since by our assumptions $q-m< p$ \  (see (\ref{csl00})\,),
we have
\begin{eqnarray}\label{fs-oxMS3}
\sum_\beta\sigma_\beta^{q-m}r_\beta^{{{n\bigl(1-\frac{q-m}p\bigr)}}}
&\overset{\mbox{\footnotesize\color{red}H\"{o}lder ineq.}}\leq &
C\biggl(\sum\limits_\beta\sigma_\beta^p\biggr)^{\frac{q-m}p}\cdot \biggl(\sum\limits_\beta r_{\beta}^{n}\biggr)^{\frac{p-q+m}{p}}\nonumber\\
&\leq& C
\sigma^{q-m}\cdot \biggl(\Lb^n(E)\biggr)^{\frac{p-q+m}{p}}.
\end{eqnarray}
The estimates of the second sum are again handled by consideration
of two separate cases.

{\bf Case I.} \ $q\ge {p}$. Then
\begin{equation}\label{le-Iomm2}
\sum_\beta
 \sigma_\beta^qr_\beta^{(k+\al-\frac{n}p)q+\mu}\overset{(\ref{eles1})}\le
\sum_\beta\sigma^{p}_\beta\le\sigma^p.
\end{equation}

{\bf Case II.}  \ $q< {p}$.  From the definition~(\ref{mmmm1}) we have the following elementary identity:
\begin{equation}\label{fimm3}
\bigl(k+\al-\frac{n}p\bigr)q+\mu=n(1-\frac{q}p)+m(k+\al-1)>n\bigl(1-\frac{q}p\bigr).
\end{equation}
 Then
\begin{eqnarray}\label{le-Iomm3}
\sum_\beta
 \sigma_\beta^qr_\beta^{(k+\al-\frac{n}p)q+\mu}\le \sum_\beta
 \sigma_\beta^qr_\beta^{n(1-\frac{q}p)}
&\overset{\mbox{\footnotesize\color{red}H\"{o}lder ineq.}}\leq &
\biggl(\sum\limits_{\beta}\sigma_\beta^{p}\biggr)^{\frac{q}{p}}\cdot \biggl(\sum\limits_{\beta}r_{\beta}^{n}\biggr)^{\frac{{p}-q}{{p}}}\nonumber\\
&\leq & \sigma^q\delta^{\frac{{p}-q}{{p}}}.
\end{eqnarray}
Now for both cases (I) and (II) we have by
(\ref{oxMS1})--(\ref{le-Iomm3}) that $\Phi(E)\le \omega(\delta)$, where
the function $\omega(\delta)$ satisfies the condition
$\omega(\delta)\searrow 0$ as $\delta\searrow 0$. The lemma is proved.
\end{proof}

\

{\sc Step III (finishing of the proof of Theorem~\ref{DFT-F}~(iii) for the case $k\le n$.}

\medskip

From Theorem~\ref{Th_ap} it follows that for any $\e>0$ there
exists a decomposition $\R^n=U\cup E$, \,$E=\R^n\setminus U$,
where $\meas(U)<\e$ \,and \,the identities $v=h$ and $\nabla
v=\nabla h$ hold on the set~$E$, where the mapping $h$ belongs to
the class $\Cc(\R^n,\R^d)$. By Theorem~\ref{DST-H}, proved in the
previous subsection, we have $\Phi(Z_v\cap E)=\Phi_h(Z_h\cap
E)=0$. On the other hand, by Lemma~\ref{cor00-sl} the value
$\Phi(Z_v\cap U)$ could be made arbitrary
small. Therefore, $\Phi(Z_v)=0$ as required.

\medskip

\subsection{Bridge F.-D. Theorem for mappings of  potential spaces~$\Le^{k+\al}_p(\R^n,\R^d)$: case $k>n$.}\label{DTSC2}

The proof for this case could be done almost word by word the same way as in the~previous subsection with the following evident difference:
on the Step~1 for the estimates on a single cube we need to use Theorem~\ref{Sob-est-1} instead of Theorem~\ref{Sob-est-2}.
Thus the proof now is even much easier since on the right hand side of  the estimate of  Theorem~\ref{Sob-est-1} we have only one term (instead of two terms in the estimate of
Theorem~\ref{Sob-est-2}\,).

\medskip

\subsection{Bridge F.-D. Theorem for mappings of  Sobolev--Lorentz potential spaces~$\Le^{k+\al}_{p,1}(\R^n,\R^d)$ with $(k+\al)p=n$.}\label{DTSC3}

The proof for this case could be done almost word by word the same way as in previous subsection~\ref{HTFSF}
with the following evident difference:
on the Step~1 for the estimates on a single cube we need to use Theorem~\ref{Sl-est-3}~(ii) instead of Theorem~\ref{Sob-est-2}.
Also we need to use the following well-known supadditive property  for Lorentz norms:

{\sl Suppose that $1\le p<\infty$ and $E=\bigcup_{j \in \N} E_j$,
where $E_j$ are measurable and mutually disjoint subsets of
$\R^n$. Then
$$
\sum_j\|f\|^p_{\LL_{p,1}(E_j)}\le\|f\|^p_{\LL_{p,1}(E)},
$$
where by definition $\|f\|_{L_{p,1}(E)}:=\|f\cdot 1_E\|_{L_{p,1}}$, and $1_E$ denotes the indicator function of the set~$E$.}
(See, e.g., \cite{rom1} or \cite{Maly2}.\,)

\medskip
Summarizing the results, obtained in above subsections~\ref{MSP}, \ref{HTFSF} -- \ref{DTSC3},
we conclude, that the proof of Theorem~\ref{DFT-F} is finished completely.

\

\section{Appendix: estimates of the critical values on cubes}\label{appendixII}

Let $\mu\ge 0$, \,$q>0$, \,and \,$v:\R^n\to\R^d$ \,be a continuous function. For a set $E\subset\R^n$
as before define the set function
\begin{equation}\label{app-dd5}
\Phi(E)=\inf\limits_{E\subset\bigcup_j
D_j}\sum\limits_j\bigl(\diam D_j\bigr)^\mu\bigl[\diam
v(D_j)\bigr]^q,
\end{equation}
where the infimum is taken over all countable families of compact
sets $\{D_j\}_{j\in \N}$ such that $E\subset\bigcup_j D_j$. Then \ $\Phi(\cdot)$ is a countably
subadditive and the implication
\begin{equation}\label{dd6} \Phi(E)=0\ \boldsymbol{\Rightarrow}\
\biggl[\H^\mu\bigl(E\cap v^{-1}(y)\bigr)=0\quad\mbox{for
$\H^q$-almost all }y\in\R^d\biggr]
\end{equation}holds (see Theorem~\ref{FubN}\,).

Our purpose here is to estimate $\Phi$ for subsets of critical set in cubes for different classes of mappings.

For all 
the following four subsections  fix $m\in\{0,\dots,n-1\}$ \,and \,$d\ge m$. Take also a positive parameter $q\ge m$  and nonnegative~$\mu\ge0$ required in the definition of the set--function
$\Phi$. 

For a regular (in a sense) 
mapping~$v:\R^n\to\R^d$ denote as before
$$Z_{v,m}=\{x\in\R^n\setminus A_v:\rank\nabla v(x)\le m\}.$$ Here $A_v$ means the set of `bad' points, where $v$ is not differentiable 
or which are not Lebesgue points for~$\nabla v$ (of course, $A_v=\emptyset$ if the gradient~$\nabla v$ is a~continuous function\,).  
It is convenient (and sufficient for our purposes) to restrict our attention on the following subset of critical points
\begin{equation}\label{Z'}Z'_v=\{x\in Z_{v,m}:|\nabla v(x)|\le1\}.
\end{equation}

\subsection{Estimates on cubes for Holder classes of mappings. }\label{HT-es}

Fix \,$k\ge1$, \,$0\le\al\le 1$, \,and
\,$v\in \Cc(\R^n,\R^d)$. By definition of the space
$\Cc$ and since the result has the~local nature, we may assume without
loss of generality that
\begin{eqnarray}\label{app-hhc1} &&\mbox{$|\nabla^kv(x)-\nabla^kv(y)|\le  |x-y|^\alpha$ \ for all $x,y\in\R^n$.}\\
\label{app-hhc2} &&|\nabla v(x)|\le 1\qquad \mbox{for all $x\in\R^n$.}
\end{eqnarray}

The main result of this subsection is contained in the following

\begin{ttt}\label{Hold-est-1}{\sl
 Under above assumptions, for any sufficiently small
$n$-dimensional interval $Q\subset\R^n$ the estimate
\begin{equation}
\label{app-6} \Phi(Q\cap Z_{v,m})\leq C\,\ell(Q)^{q+\mu+(k+\al-1)(q-m)}
\end{equation}
holds, where the constant $C$ depends on $n,m,k,\alpha,d$ only.}
\end{ttt}

\begin{proof}
 Let the assumptions in the beginning of this subsection are fulfilled. Fix an $n$-dimensional  interval $Q\subset\R^n$ of size $\ell(Q)<\frac1{\sqrt{n}}$.
 Without loss of generality we may assume that
the origin~$0\in Q$. Take the polynomial $P$ of degree $k$ such that
\begin{equation}
\label{h-q1} \nabla^jv(0)=\nabla^jP(0)\qquad\forall j=0,1,\dots,k.
\end{equation}
Denote $v_Q=v-P$.
Then from assumption~(\ref{app-hhc1}) we have
\begin{eqnarray}\label{pr-h-1} &&\mbox{$|\nabla^kv_Q(x)|\le  r^\alpha$\qquad \ \ \ \ for all $x\in Q$.}\\
\label{pr-h2} &&|\nabla v_Q(x)|\le   r^{k+\al-1}\qquad \mbox{for all $x\in Q$,}
\end{eqnarray}
where we denote for convenience $r=\sqrt{n}\,\ell(Q)$. Put $\e= r^{k+\al-1}$. Since by our choice  $\ell(Q)\le\frac1{\sqrt{n}}$,
 we have in particular that
\begin{equation}
\label{itt2} \e<1.
\end{equation}
Denote $Z_v=Q\cap Z_{v,m}$. Since $\nabla P_{Q}(x)=\nabla v(x)-\nabla v_{Q}(x)$, $|\nabla
v_Q(x)|\le \e$, $|\nabla v(x)|\leq 1$, and $\lambda_{m+1}(v,x)=0$ for
$x\in Z_v$ , we have\footnote{Here we use the following
elementary fact: for any linear maps $L_1 \colon \R^n\to\R^d$ and
$L_2 \colon \R^n\to\R^d$ the estimates $\lambda_l(L_1+L_2)\le
\lambda_l(L_1)+\|L_2\|$ hold for all $l=1,\dots,d$, see, e.g.,
\cite[Proposition 2.5 (ii)]{Yom}.}
$$
Z_v\subset \bigl\{x\in Q: \lambda_1(P_{Q},x)\le
1+\e,\dots,\lambda_{m}(P_{Q},x)\le 1+\e,\ \lambda_{m+1}(P_{Q},x)\le
\e\bigr\}.
$$
Applying Theorem~\ref{lb8} to polynomial $P$ and
using the Lipschitz condition $\diam v_Q(Q)\le\e\,r$, we
find a finite family of balls $T_j\subset\R^d$, $j=1,\dots,N$ with
$N \leq C_{Y}(1+\e^{-m})$, each of radius $2\e r$,
such that
$$
\bigcup\limits_{j=1}^{N}T_j \supset v(Z_v).
$$
Therefore, we have
\begin{equation*}\label{Iom1}
\Phi(Z_{v})\le
C\,N\e^{q}r^{q+\mu}\le
C_1(1+\e^{-m})\,\e^q\,r^{q+\mu}\overset{\footnotesize(\ref{itt2})}\le
C\,\e^{q-m}\,r^{q+\mu}.
\end{equation*}
The last formula, because of definition of~$\e$, implies the required estimate~(\ref{app-6}).
The Theorem is
proved.
\end{proof}

The analysis of this simple proof shows, that if we replace the condition~(\ref{app-hhc1}) by more general assumption
$$\mbox{$|\nabla^kv(x)-\nabla^kv(y)|\le  A\,\ell(Q)^\alpha$ \ for all $x,y\in Q$},$$
then instead of~(\ref{app-6}) the modified estimate 
\begin{equation}
\label{app-6-new} \Phi(Q\cap Z_{v,m})\leq C\,A^{q-m}\,\ell(Q)^{q+\mu+(k+\al-1)(q-m)}
\end{equation}
holds, where again the constant $C$ depends on $n,m,k,\alpha,d$ only.

\subsection{Estimates on cubes for Sobolev classes of mappings.  Case~I: \,$(k+\al-1)p>n$.  }\label{SC-I}

Fix \,$k\ge1$, \,$0\le\al< 1$, $1<p<\infty$, \,and
\,$v\in \Ll(\R^n,\R^d)$. In this subsection we consider  the case, when
\begin{equation}\label{app-shhc1''} (k+\al-1)p>n,
\end{equation}
i.e., when the gradient $\nabla v$ is a continuous and uniformly bounded function by Sobolev Imbedding Theorems. For convenience, in this section we assume that
$$\sup\limits_{x\in \R^n}|\nabla v(x)|\le 1.$$

\begin{ttt}\label{Sob-est-1}{\sl
 Under above assumptions, there exists a function~$h\in L_p(\R^n)$ (depending on~$v$ only\,) such that for any sufficiently small
$n$-dimensional  interval $Q\subset\R^n$ the estimate
\begin{equation}
\label{app-s-6} \Phi(Q\cap Z_{v,m})\leq C\,\sigma^{q-m}\ell(Q)^{q+\mu+(k+\al-1-\frac{n}p)(q-m)}
\end{equation}
holds, where \begin{equation}
\label{app-s-7} \sigma=\|h\|_{L_{p}(Q)},
\end{equation}
and  the constant $C$ depends on $n,m,k,\alpha,d,p$ only.}
\end{ttt}

\begin{proof}
The proof here is very similar to the proof of Theorem~\ref{Hold-est-1} from the previous subsection. But in the beginning we have to obtain the uniform estimates
for the gradient of the difference between the function and some polynomial.

 Let the assumptions in the beginning of this subsection are fulfilled. 
 Put $l=\min\{i\in\Z_+:k+\al-n-2<i\}$. In particular, $l\ge0$. Denote $\kk=k-l$. Then
by construction
\begin{equation}\label{s-imp} \kk+\al-n-2< 0.
\end{equation}
We claim also, that 
\begin{equation}\label{app-shhc1} (\kk+\al-1)p>n.
\end{equation}
Indeed, if $k+\al-n-2<0$, then $l=0$, $\kk=k$, and the  inequality~(\ref{app-shhc1}) follows immediately from the assumption~(\ref{app-shhc1''}). In the other hand,
if $k+\al-n-2\ge 0$, then $l>0$ and by construction we have
$l-1\le k+\al-n-2,$
that is equivalent
$\kk+\al-n-1\ge0,$ and the inequality~(\ref{app-shhc1}) follows from the assumption~$p>1$.

Put  $u=\nabla^l v$. From the inclusion~$u\in\Le^{\kk+\al}_{p}(\R^n,\R^d)$  it follows that $u=\GG_{\kk+\alpha}(g):=G_{\kk+\al}*g$ for some $g\in L_{p}(\R^n)$.
We will use the Hardy--Littlewood maximal function $\M g$.  The well-known properties of maximal functions (see, e.g., \cite{St}\,) imply
\begin{equation*}\label{lll-base2}
h:=\M g\in L_{p}(\R^n).\end{equation*}
 
 Fix an $n$-dimensional  interval $Q\subset\R^n$ of size~$r=\ell(Q)\le1$.
From Remark~\ref{rem-SCE-1} (see formula~(\ref{pr-s-12})\,), applied to the function~$u=G_{\kk+\al}*g$,  and from assumption~(\ref{app-shhc1}),
it follows  that there exists a polynomial~$\tilde P$ of degree~$\kk$ such that
\begin{equation}
\label{app-pr-s-9} \sup\limits_{x\in Q}|\nabla u_Q(x)|\le  C\,\|\M g\|_{L_p(Q)}\,r^{\kk+\al-1-\frac{n}p}  ,
\end{equation}
where $u_Q=u-\tilde P$. Since $u=\nabla^l v$, we have, that for some polynomial $P$ of degree~$k$ the estimate
\begin{equation}
\label{app-pr-st2} \sup\limits_{x\in Q}|\nabla v_Q(x)|\le  r^l\sup\limits_{x\in Q}|\nabla u_Q(x)|\le  C\,\|\M g\|_{L_p(Q)}\,r^{k+\al-1-\frac{n}p} .
\end{equation}
Denote the right hand side of the last inequality by~$\e$. Of course, taking $r=\ell(Q)$ sufficiently small, we could assume without loss of generality that
$\e<1$.

From this moment we could repeat almost word by word the last part of the proof of Theorem~\ref{Hold-est-1} to obtain the required estimate~(\ref{app-s-6}).
\end{proof}

\subsection{Estimates on cubes for Sobolev classes of mappings.  Case~II: \,$n-p<(k+\al-1)p<n$.  }\label{SC-II}

Fix \,$k\ge1$, \,$0\le\al< 1$, $1<p<\infty$, \,and
\,$v\in \Ll(\R^n,\R^d)$. In this subsection we consider  the case, when $k+\al>1$ and
\begin{equation}\label{app-3-shhc1} n-p<(k+\al-1)p<n,
\end{equation}
i.e., when $v$  a continuous function by Sobolev Imbedding Theorems, but the gradient $\nabla v$ could be discontinuous.

From the inclusion~$v\in\Le^{k+\al}_{p}(\R^n,\R^d)$  it follows that $\nabla^kv=\GG_\alpha(g):=G_\al*g$ for some $g\in L_{p}(\R^n)$.
We will use the Hardy--Littlewood maximal function $\M g$.  The well-known properties of maximal functions (see, e.g., \cite{St}\,) imply
\begin{equation*}\label{lll-base2}
\M g\in L_{p}(\R^n).\end{equation*}

\begin{ttt}\label{Sob-est-2}{\sl
 Under above assumptions,  if an addition
 \begin{equation}\label{s-as}q+\mu>\tt:=n-(k+\al-1)p,
\end{equation}
then for any
$n$-dimensional interval $Q\subset\R^n$ of size~$r=\ell(Q)$ the estimate
\begin{equation}
\label{app-s-9} \Phi(Z'_v\cap Q)\leq C\,\biggl(\,\sigma^q r^{(k+\al-\frac{n}p)q+\mu}\ +\ \sigma^{q-m}r^{q+\mu+(k+\al-1-\frac{n}p)(q-m)}\biggr),
\end{equation}
holds, where \begin{equation}
\label{app-s-10}\sigma=\|\M g\|_{L_{p}(Q)},
\end{equation}
and  the constant $C$ depends on $n,m,k,\alpha,d,p$ only.}
\end{ttt}

\begin{proof}
 Let the assumptions in the beginning of this subsection are fulfilled. Fix an $n$-dimensional  interval $Q\subset\R^n$  of size $r=\ell(Q)$. Take
 the polynomial $P$ of degree~$k$ from Lemma~\ref{SCE-1} such that
\begin{equation}
\label{s-uk1} |\nabla v_Q(x)|\le C\int\limits_{Q}\frac{\M g(y)}{|x-y|^{n-k-\al+1}}\,dy\ \qquad\forall x\in Q,
\end{equation}
where $v_Q=v-P$. 

 Put
\begin{equation}
\label{sc2-14'}\e_*=\|\M g\|_{L_p(Q)}\,r^{k+\al-1-\frac{n}p}.
\end{equation}
We emphasize that now we could not assume that $\e_*$ is small since the exponent~${k+\al-1-\frac{n}p}$ is
{\it strictly less} than zero.

Denote
$$
\sigma=\|\M g\|_{\LL_p(Q)},
$$
and for each $j\in \mathbb Z$ define
$$
E_j=\bigl\{x\in Q: \M_Q |\nabla v_{Q}|(x)\in(2^{j-
1},2^{j}]\bigr\} \quad \mbox{ and } \quad
\delta_j=\H^{q+\mu}_\infty(E_j),
$$
where $\M_Q$ is the modified Hardy--Littlewood maximal function  defined by formula~(\ref{lip-m}).
Put
$$s=\frac{(q+\mu)p}{n-(k+\al-1)p},\qquad\tau=\frac{s}p\bigl(n-(k+\al-1)p\bigr)=q+\mu.$$
From these definitions it follows, in particular, that
\begin{equation}
\label{e-s-1} \biggl(\frac\sigma{\e_*}\biggr)^s=r^{q+\mu},
\end{equation}
further, since $n<(k+\al)p$, we have
\begin{equation}
\label{e-s-1'} s>q+\mu,
\end{equation}
moreover, since by Theorem assumption $q+\mu>n-(k+\al-1)p$, we have
\begin{equation}
\label{e-s-1''} s>p.
\end{equation}
Then by Theorem~\ref{AM1} (applied for the case $\beta=(k+\al-1)$\,),
\begin{equation}
\label{tr-d1} \sum\limits_{j=-\infty}^\infty \delta_j2^{js}\leq
C\sigma^s
\end{equation}
for a constant $C$ depending on $n,d,\beta,\tau,s$ only. By the definition
of the~Hausdorff measure, for each $j\in\mathbb Z$ there exists a
family of balls $B_{ij}\subset\R^n$ of radii $r_{ij}$ such that
\begin{equation}
\label{res1} E_j\subset \bigcup\limits_{i=1}^\infty B_{ij} \mbox{\
\ \ and\ \ \ }\sum\limits_{i=1}^\infty r^{{q+\mu}}_{ij}\le c\,\delta_j.
\end{equation}

Denote
$$
Z_j=Z'_v\cap E_j \quad \mbox{ and } \quad Z_{ij}=Z_j\cap B_{ij}.
$$
By construction $Z'_v\cap Q=\bigcup_{j}Z_{j}$ \ and \
$Z_j=\bigcup_{i}Z_{ij}$.

Take an integer value $j_*$ such that $\e_*\in(2^{j_*-1},2^{j_*}]$. Denote
$Z_*=\bigcup_{j< j_*}Z_{j}$, \ $Z_{**}=\bigcup_{j\ge j_*}Z_{j}$.
Then by construction
$$
Z'_v\cap Q=Z_*\cup Z_{**},\quad Z_*\subset \{x\in Z'_v\cap Q:\M_Q
|\nabla v_{Q}|)(x)<\e_*\}.
$$
Further, since  $\nabla P_{Q}(x)=\nabla v(x)-\nabla v_{Q}(x)$, $|\nabla
v_Q(x)|\le 2^j$, $|\nabla v(x)|\leq 1$, and
$\lambda_m(v,x)=0$ for $x\in Z_{ij}$ , we have\footnote{Here we
use the following elementary fact: for any linear maps $L_1 \colon
\R^n\to\R^d$ and $L_2 \colon \R^n\to\R^d$ the estimates
$\lambda_l(L_2+L_2)\le \lambda_l(L_1)+\|L_2\|$ hold for all
$l=1,\dots,d$, see, e.g., \cite[Proposition 2.5 (ii)]{Yom}.}
$$
Z_{ij}\subset \bigl\{x\in B_{ij}: \lambda_1(P_{Q},x)\le
1+2^{j},\dots,\lambda_{m}(P_{Q},x)\le 1+2^{j},\
\lambda_{m+1}(P_{Q},x)\le 2^{j}\bigr\}.
$$
From definition  of $E_j$ and from Lemma~\ref{l-lip} (applying to the function~$v_Q$\,)  we have
$$|v_Q(x)-v_Q(y)|\le C_M2^jr_{ij}\qquad\forall x,y\in Z_{ij}.$$
From this fact, applying Theorem~\ref{lb8} to polynomial
$P_{Q}$ with $B=B_{ij}$ and
$\e=\e_j=2^{j}$, we find a finite family of balls
$T_\nu\subset\R^d$, $\nu=1,\dots,\nu_j$ with $\nu_j \leq
C_{Y}(1+\e_j^{-m})$, each of radius $(1+C_{M})\e_jr_{ij}$,  such
that
$$
\bigcup\limits_{\nu=1}^{\nu_j}T_\nu \supset v(Z_{ij}).
$$
Therefore, for every $j\ge  j_*$ we have
\begin{equation}\label{Iom1}
\Phi(Z_{ij})\le C_1
\nu_j\e_j^{q}r^{q+\mu}_{ij}=C_2(1+\e_j^{-m})2^{jq}r_{ij}^{q+\mu}\le
C_2(1+\e_*^{-m})2^{jq}r_{ij}^{q+\mu},
\end{equation}
where all the constants $C_\mu$ above depend on $n,m,k,d$ only.
By the same reasons, but this time applying Theorem~\ref{lb8} and
Lemma~\ref{l-lip} with $\e=\e_*$ and instead of the balls $B_{ij}$
taking the big  ball $B\supset Q$ with radius $\sqrt{n}r$, we have
\begin{equation}
\Phi(Z_{*})\le C_3(1+\e_*^{-m})\e_*^q r^{q+\mu}.\label{Iom2}
\end{equation}
The right hand side of the last formula, from the definition of ~$\e_*=\sigma\,r^{k+\al-1-\frac{n}p}$, is equivalent to the right hand side of the required
estimate~(\ref{app-s-9}).

From (\ref{Iom1}) we get immediately
\begin{eqnarray}\label{Iom2'}
\Phi(Z_{**})\le C_2(1+\e_*^{-m})\sum\limits_{j\ge  j_*}\sum\limits_i2^{jq}r_{ij}^{q+\mu} \overset{(\ref{res1})}\leq C_3(1+\e_*^{-m})\sum\limits_{j\ge  j_*}
2^{jq}\delta_j\nonumber\\ \!\!\!\!\!\!\!\!\!\!\!\!\!\!\!\!\!\!\!\!\!\!\!\!\!\!\!\!\overset{\footnotesize(\ref{e-s-1'})}\le C_3(1+\e_*^{-m}) 2^{j_*(q-s)}\sum\limits_{j\ge  j_*}
2^{js}\delta_j\overset{\footnotesize(\ref{tr-d1})}\le  C_4(1+\e_*^{-m}) \e_*^{(q-s)}\sigma^s \overset{\footnotesize(\ref{e-s-1})}\le C_5
(1+\e_*^{-m}) \e_*^{q}r^{q+\mu}.
\end{eqnarray}
The right hand side of the last formula, by the same reasons (see the commentary after~(\ref{Iom2}) \,)
is equivalent to  the right hand side of the required
estimate~(\ref{app-s-9}).

Thus from~(\ref{Iom2})--(\ref{Iom2'}) the required estimate~(\ref{app-s-9}) follows. The Lemma is proved.
\end{proof}

\subsection{Estimates on cubes for Sobolev--Lorentz classes of mappings: the general case \,$(k+\al)p\ge n$.  }\label{SC-III}

Fix \,$k\ge1$, \,$0\le\al< 1$, $1<p<\infty$, \,and
\,$v\in \Lll(\R^n,\R^d)$. In this subsection we consider  the case, when $k+\al>1$ and
\begin{equation}\label{app-sl-1} (k+\al)p\ge n,
\end{equation}
i.e., when $v$ is  a continuous function (see, e.g., \cite{KK3}\,), but the gradient $\nabla v$ could be discontinuous in general (if $(k+\al-1)p<n$\,).

\begin{ttt}\label{Sl-est-3}{\sl
 Under above assumptions, there exists a~function $h\in L_{p,1}(\R^n)$ (depending on~$v$\,) such that the following statements are fulfilled:

 \begin{itemize}
\item[(i)] \,if  $(k+\al-1)p\ge n$, then gradient $\nabla v$ is continuous and uniformly bounded function, and
for any sufficiently small
$n$-dimensional interval $Q\subset\R^n$ the estimate
\begin{equation}
\label{t-sl-2}\Phi(Z'_v\cap Q)\leq C\,\sigma^{q-m}r^{q+\mu+(k+\al-1-\frac{n}p)(q-m)}
\end{equation}
holds, where again \begin{equation}
\label{sl-app-s-10} r=\ell(Q),\qquad\sigma=\|h\|_{L_{p,1}(Q)}.
\end{equation}
and  the constant $C$ depends on $n,m,k,\alpha,d,p$ only.

\item[(ii)] \,if  $n-p\le(k+\al-1)p< n$, then under additional assumption
 \begin{equation}\label{sl-as}q+\mu\ge\tt:=n-(k+\al-1)p
\end{equation}for any
$n$-dimensional interval $Q\subset\R^n$ the estimate
\begin{equation}
\label{sl-app-s-9}\Phi(Z'_v\cap Q)\leq C\,\biggl(\,\sigma^q r^{(k+\al-\frac{n}p)q+\mu}+\sigma^{q-m}r^{q+\mu+(k+\al-1-\frac{n}p)(q-m)}\biggr)
\end{equation}
holds with the same $\sigma,r$.
\end{itemize}}
\end{ttt}

\begin{rem}\label{rem-sl-1}
Formally estimates in Theorem~\ref{Sl-est-3} are the same as in Theorems~\ref{Sob-est-1}--\ref{Sob-est-2}, the only difference  is in the definition of~$\sigma$ (using
the~Lorentz norm instead of Lebesgue one, cf. formulas~(\ref{app-s-7}) and (\ref{sl-app-s-10})\,).
However, Theorem~\ref{Sl-est-3} is `stronger' in  a~sense than the previous Theorems~\ref{Sob-est-1}--\ref{Sob-est-2}. Namely, there are
 some important (limiting) cases, which are not covered by Theorem~\ref{Sob-est-2}, but one could still apply the Theorem~\ref{Sl-est-3} for these cases.
It happens for the following values of the parameters:
  \begin{equation}\label{sl-par1}(k+\al)p=n,
\end{equation}
or \begin{equation}\label{sl-par2}(k+\al-1)p=n,
\end{equation}
or
\begin{equation}\label{sl-par3}q+\mu=\tt.
\end{equation}
It means, that the Lorentz norm is a sharper and more accurate tool here than the Lebesgue norm.
\end{rem}

\begin{proof}[Proof of Theorem~\ref{Sl-est-3}]
As in the previous subsection~\ref{SC-I} put $l=\min\{i\in\Z_+:k+\al-n-2<i\}$. In particular, $l\ge0$. Denote $\kk=k-l$ and $u=\nabla^l v$. Then
by construction
\begin{equation}\label{sll-imp} u\in\Le^\kk_p(\R^n),\qquad \kk+\al-n-2< 0,
\end{equation}
moreover,
\begin{equation}\label{sll-app-shhc1} (\kk+\al-1)p\ge n \qquad\mbox{ if }\quad (k+\al-1)p\ge n
\end{equation}
(see the above  discussion around~(\ref{s-imp})--(\ref{app-shhc1})\,). \ From the inclusion~$u\in\Le^{\kk+\al}_{p,1}(\R^n,\R^d)$  it follows that $u=\GG_{\kk+\alpha}(g):=G_{\kk+\al}*g$ for some $g\in L_{p}(\R^n)$.
We will use the Hardy--Littlewood maximal function $\M g$.  Recall, that by properties of Lorentz spaces, the standard
estimate
\begin{equation*}\label{lll-base2}
 \|\M g\|_{L_{p,1}}\le
C\,\|g\|_{L_{p,1}}\end{equation*}
holds for the considered case  $1<p<\infty$  (see, e.g.,  \cite[Theorem 4.4]{Maly2}\,). Put $h=\M g$.

The proof of Theorem \ref{Sl-est-3}  is very similar to that one of Theorems~\ref{Sob-est-1}--\ref{Sob-est-2}: the main differences concern the limiting cases
(\ref{sl-par1})--(\ref{sl-par3}) mentioned above.

\begin{itemize}
\item{\sc Case $(k+\alpha-1)p >n$.} The required assertion in~(i) follows immediately from Theorem~\ref{Sob-est-1} and from  the well-known inequality
$$\|f\|_{L_{p}(Q)}\le\|f\|_{L_{p,1}(Q)}.$$

\item{\sc Case $(k+\alpha-1)p =n$.} \  Then by above notations $l=0$ and $v=u=G_{k+\al}*g$. Fix an $n$-dimensional  interval $Q\subset\R^n$  of size~$r=\ell(Q)$.
 Take
 the polynomial $P$ of degree~$k$ from Theorem~\ref{SCE-1} such that
\begin{equation}
\label{sll-uk1} |\nabla v_Q(x)|\le C\int\limits_{Q}\frac{\M g(y)}{|x-y|^{n-k-\al+1}}\,dy\ \qquad\forall x\in Q,
\end{equation}
where $v_Q=v-P$. \ Put $\sigma=\|\M g\|_{L_{p,1}(Q)}$.

Put $\beta=(k+\al-1)$. From the generalised Holder inequality for Lorentz norms
$$\int\limits_Q\frac{\M g(y)}{|y-x|^{n-\beta}}\,dy\le \|\M g\|_{L_{p,1}(Q)}\cdot\biggl\|\frac{1_Q}{|\cdot-x|^{n-\beta}}
\biggr\|_{L_{\frac{p}{p-1},\infty}}= C\, \|g\|_{L_{p,1}}$$
(see, e.g., \cite[Theorem~3.7]{Maly2}\,) and from (\ref{sll-uk1})  it follows immediately that
\begin{equation}
\label{s2-uk3} \sup\limits_{x\in Q}|\nabla v_Q(x)|\le C\,\|\M g\|_{L_{p,1}(Q)}= C\,\sigma.
\end{equation}
Denote the right hand side of the last inequality by~$\e$. Of course, taking $r=\ell(Q)$ sufficiently small, we could assume without loss of generality that
$\e<1$.
Thus to finish the prove for this case
we could repeat almost word by word the last part of the proof of Theorem~\ref{Hold-est-1} to obtain the inequality
\begin{equation}
\label{fsl1}\Phi(Z'_v\cap Q)\leq C\,\sigma^{q-m}r^{q+\mu},
\end{equation}
which is equivalent to the required estimate~(\ref{t-sl-2}) for the considered values of~$k,\al,p$.

\item{\sc Case $n-p\le(k+\alpha-1)p <n$.} \  Then again $l=0$, $u=v$,  and for this last case the required assertion~(ii) can be proved repeating almost "word by word" the same arguments as in the previous
Theorem~\ref{Sob-est-2} with
the following evident modification:  now it is possible that $s=p$ (respectively, $q+\mu=\tt$\,), and for this situation  one has to apply Theorem~\ref{lb7} 
instead of previous Theorem~\ref{AM1} (where $s>p$\,).
\end{itemize}
\end{proof}

\subsection{Estimates on cubes for Sobolev classes of mappings $W^k_1(\R^n)$, \ $k\ge n$.  }\label{SC-IIII}
As usual, by $W^k_p(\R^n,\R^d)$ we denote the space of functions $v:\R^n\to\R^d$ such that all its generalized derivatives
$D^jv$,  \,$j=1,\dots,k$ \ up to order~$k$ belong to the space~$L_p$. It is well known that
$$W^k_p(\R^n)=\Le^{k}_{p}(\R^n)$$ for $p>1$. In this subsection we consider the limiting case $p=1$ for Sobolev spaces $W^k_1$.
It is well known that functions from the Sobolev space $W^k_1(\R^n,\R^d)$ are continuous if
\begin{equation}
\label{cs-1}k\ge n,
\end{equation}
so we assume this condition below:
fix $k\ge n$, \,and
\,$v\in W^k_1(\R^n,\R^d)$.

Denote again $Z_{v,m}=\{x\in\R^n\setminus A_v:\rank\nabla v(x)\le m\}$. Here $A_v$ means the set of `bad' points, at which either the function~$v$
is not differentiable or which are not Lebesgue points for~$\nabla v$.
Recall, that
the set~$A_v$ is small in a sense: the equality
\begin{equation}\label{csc1}\H^{\tau_*}(A_v)=0\qquad\mbox{ for }\ \ \tt:=n-k+1
\end{equation}
holds (see, e.g., \cite{BKK2}\,); in particular, $A_v=\emptyset$ \ if \ $k-1\ge n$.

As in the previous section, it is convenient (and sufficient for our purposes) to restrict our attention on the following subset of critical points
\begin{equation}\label{csZ'-}Z'_v=\{x\in Z_{v,m}:|\nabla v(x)|\le1\}.
\end{equation}
Take also a positive parameter $q\ge m$  and nonnegative~$\mu\ge0$.

\begin{ttt}\label{Sl-est-4}{\sl
 Under above assumptions, the following statements hold:

 \begin{itemize}
\item[(i)] \,if  $k-1\ge n$, then the gradient $\nabla v$ is continuous and uniformly bounded function, and
for any sufficiently small
$n$-dimensional interval $Q\subset\R^n$ the estimate
\begin{equation}
\label{t-cs-2}\Phi(Z'_v\cap Q)\leq C\,\sigma^{q-m}r^{q+\mu+(k-1-n)(q-m)}
\end{equation}
holds, where again \begin{equation}
\label{cs-app-s-10} r=\ell(Q),\qquad\sigma=\|\nabla^kv\|_{L_{1}(Q)}.
\end{equation}
and  the constant $C$ depends on $n,m,k,d$ only.

\item[(ii)] \,if  $k=n$, then under additional assumption
 \begin{equation}\label{cs-as}q+\mu\ge 1
\end{equation}for any
$n$-dimensional interval $Q\subset\R^n$ the estimate
\begin{equation}
\label{cs-app-s-9}\Phi(Z'_v\cap Q)\leq C\,\biggl(\,\sigma^q r^{\mu}+\sigma^{q-m}r^{\mu+m}\biggr),
\end{equation}
holds with the same $r,\sigma$, and with $C$ depending on $n,m,k,d$ only.
\end{itemize}}
\end{ttt}

\begin{rem}\label{rem-cs-1}
Estimates in Theorem~\ref{Sl-est-4} are very close to the estimates in \ref{Sl-est-3}
(formally, we could obtain these estimates, taking $\alpha=0$ and $p=1$  in Theorem~ \ref{Sl-est-3} and using
the corresponding another definition for~$\sigma$,
cf. (\ref{sl-app-s-10}) and (\ref{cs-app-s-10})\,). But in Theorem~\ref{Sl-est-4} we could not use the maximal function as before: it is well known, that in general
$$\M\nabla^kv\notin L_1(\R^n).$$
\end{rem}

\begin{proof}[Proof of Theorem~\ref{Sl-est-4}]
 Let the assumptions in the beginning of this subsection are fulfilled. Fix an $n$-dimensional  interval $Q\subset\R^n$. Take the approximating polynomial~$P=P_Q$ from the~subsection~\ref{asfp} and denote $v_Q=v-P$, \ $r=\ell(Q)$. Then we have
\begin{equation*}
\label{cs-1--} \bigl\|v_Q\bigr\|_{L_\infty(Q)}\le C\,r^{k-n}\,\|
\nabla^{k} v\|_{\LL_{1}(Q)};
\end{equation*}
\begin{equation}
\label{pr1} \bigl\|\nabla
v_Q\bigr\|_{L_\infty(Q)}\le C\,r^{k-1-n}\,\|
\nabla^{k} v\|_{\LL_{1}(Q)}\qquad\mbox{\rm if \ $k\ge n+1$},
\end{equation}
 where  $C$ is a constant depending on $n,d,k$ only.
Moreover, the mapping $v_{Q}$ can
be extended from~$Q$ to the entire $\R^n$ such that the
extension (denoted again) $v_{Q}\in\WW^{k}_{1}(\R^n,\R^d)$ and
\begin{equation}
\label{sc-1'} \|\nabla^{k} v_{Q}\|_{\LL_{1}(\R^n)}\le C_0
\|\nabla^{k} v\|_{\LL_{1}(Q)},
\end{equation}
where $C_0$ also depends on $n,d,k$ only.

The rest of the proof of Theorem   is very similar to that one of Theorems~\ref{Sob-est-1}--\ref{Sl-est-3}. Consider two different cases.

\begin{itemize}
\item{\sc Case $k>n$.} Then from~(\ref{pr1})  it follows immediately that
\begin{equation}
\label{cs-uk3}\bigl\|\nabla
v_Q\bigr\|_{L_\infty(Q)}\le C\,\sigma r^{k-1-n},
\end{equation}
where $\sigma=\|\nabla^{k} v\|_{\LL_{1}(Q)}$.
Denote the right hand side of~(\ref{cs-uk3}) by $\e$.
Now to finish the prove for this case
we could repeat almost word by word the last part of the proof of Theorem~\ref{Hold-est-1} to obtain the required estimate~(\ref{t-cs-2}).

\item{\sc Case $k=n$.} \  Now from Theorem~\ref{cs-lb7} and (\ref{sc-1'}) we have
\begin{equation}
\label{cs-cb3}
\int_0^\infty\H^{1}_\infty( \{x\in Q : \M \bigl(\nabla
v_Q\bigr)(x)\ge t \})\,\dd t \le C\,\Vert \nabla^k
v\Vert_{\LL_{1}(Q)}=C\,\sigma.
\end{equation}
Then for this last case the required assertion~(ii) can be proved repeating almost "word by word" the same arguments as in the previous
Theorem~\ref{Sob-est-2} with many obvious simplifications (now we have parameters $\al=0$, $s=p=1$, etc.), and with
the following evident modification:  one has to apply formula~(\ref{cs-cb3})
instead of previous Theorem~\ref{AM1}.
\end{itemize}
\end{proof}

\noindent Dipartimento di Matematica e Fisica Universit\`a degli studi della
Campania "Luigi Vanvitelli," viale Lincoln 5,
81100, Caserta, Italy\\
e-mail: {\it Adele.FERONE@unicampania.it}

\bigskip

\noindent School of Mathematical Sciences
Fudan University, Shanghai 200433, China,
and Voronezh State University, Universitetskaya pl. 1,
Voronezh, 394018, Russia\\
e-mail: {\it korob@math.nsc.ru}

\bigskip

\noindent Dipartimento di Matematica e Fisica Universit\`a degli studi della
Campania "Luigi Vanvitelli," viale Lincoln 5,
81100, Caserta, Italy\\
e-mail: {\it alba.roviello@unicampania.it}

\end{document}